\documentclass[a4paper,leqno,10pt]{amsart}

\usepackage{epsf,graphicx,epsfig}
\usepackage{amscd}
\usepackage{amsmath,latexsym,amssymb,amsthm}
\usepackage[nospace,noadjust]{cite}
\usepackage{textcomp}
\usepackage{setspace,cite}
\usepackage{lscape,fancyhdr,fancybox}
\usepackage{stmaryrd}
\usepackage[all,cmtip]{xy}
\usepackage{tikz}
\usepackage{cancel}
\usetikzlibrary{shapes,arrows,decorations.markings}
\usepackage{graphicx}

\usepackage{color}
\usepackage{url}
\usepackage{enumerate}
\usepackage[mathscr]{euscript}

  \theoremstyle{plain}

\swapnumbers
    \newtheorem{theorem}{Theorem}[section]
    \newtheorem{proposition}[theorem]{Proposition}

    \newtheorem{subsec}[theorem]{}
\theoremstyle{definition}
    \newtheorem{definition}[theorem]{Definition}

\theoremstyle{remark}

\title{}
\author{}
\date{}
\usepackage{amssymb}

\usepackage{hyperref}
\hypersetup{
	colorlinks,
	citecolor=blue,
	filecolor=black,
	linkcolor=blue,
	urlcolor=black
}

\begin{document}
\title[Extension theory of Rota-Baxter algebras and dendriform algebras]{A unified extension theory of Rota-Baxter algebras, dendriform algebras, and a fundamental sequence of Wells}

\author{Apurba Das}
\address{Department of Mathematics,
Indian Institute of Technology Kharagpur, Kharagpur-721302, West Bengal, India.}
\email{apurbadas348@gmail.com, apurbadas348@maths.iitkgp.ac.in}

\author{Nishant Rathee}
\address{Harish-Chandra Research Institute,  A CI of Homi Bhabha National Institute, Chhatnag Road, Jhunsi, Prayagraj - 211019, India.}
\email{nishant@hri.res.in}


%


%



\begin{abstract}
A Rota-Baxter algebra $A_R$ is an algebra $A$ equipped with a distinguished Rota-Baxter operator $R$ on it. Rota-Baxter algebras are closely related to dendriform algebras introduced by Loday. In this paper, we first consider the non-abelian extension theory of Rota-Baxter algebras and classify them by introducing the non-abelian cohomology. Next, given a non-abelian extension $0 \rightarrow B_S \rightarrow E_U \rightarrow A_R \rightarrow 0$ of Rota-Baxter algebras, we construct the Wells type exact sequences and find their role in extending a Rota-Baxter automorphism $\beta \in \mathrm{Aut}(B_S)$ and lifting a Rota-Baxter automorphism $\alpha \in \mathrm{Aut}(A_R)$ to an automorphism in $\mathrm{Aut}(E_U)$. We end this paper by considering a similar study for dendriform algebras.
\end{abstract}


\maketitle



\noindent

\thispagestyle{empty}

\tableofcontents

\vspace{0.2cm}

\section{Introduction}
Rota-Baxter operators and their relations with dendriform algebras is an active area of research. The notion of the Rota-Baxter operator first appeared in the work of Baxter in the fluctuation theory of probability \cite{baxter}. Subsequently, such operator and its variants are explored by some well-known mathematicians like Atkinson, Cartier and Rota, among others \cite{atkinson,cartier,rota}. In the last twenty years, Rota-Baxter operators on algebras pay much more attention due to their connection with combinatorics, splitting of algebras, Yang-Baxter equations, algebraic operads and quantum field theory \cite{aguiar,aguiar-pre-lie,connes,bai,guo-keigher,guo-book}. A Rota-Baxter algebra $A_R$ is an algebra $A$ equipped with a distinguished Rota-Baxter operator $R: A \rightarrow A$. On the other hand, Loday introduced dendriform algebras in the periodicity phenomenons in algebraic $K$-theory \cite{dialgebra}. They are also related to combinatorics (such as binary trees, shuffles etc.) and splitting of algebras \cite{lod-val-book,ronco-dend}. In \cite{aguiar,aguiar-pre-lie} Aguiar observed that a Rota-Baxter algebra naturally induces a dendriform algebra structure. This yields a functor from the category of Rota-Baxter algebras to the category of dendriform algebras.

\medskip

Algebraic structures are better understood by their extension theory and cohomology theory. Recently, such theories for Rota-Baxter algebras and dendriform algebras are studied in \cite{das-dend,DasSK,DasSKbimod,das-mishra-hazra,guodong-zhou,jiang-sheng,laza-sheng}. However, in the above-mentioned references, only the abelian extension theories are considered. There are some other type of extensions, such as central extensions and non-abelian extensions. Among all, non-abelian extension is the more general one and it unifies others. The notion of non-abelian extension theory goes back to Eilenberg and Maclane, who first considered such theory for abstract groups \cite{eilen}. Subsequently, non-abelian extensions for (super) Lie algebras, Leibniz algebras and associative are studied in \cite{casas,fial-pen,inas,liu-sheng-wang,yael,gouray}. Our first aim here is to present the non-abelian extension theory for Rota-Baxter algebras. A non-abelian extension of a Rota-Baxter algebra $A_R$ by another Rota-Baxter algebra $B_S$ is a short exact sequence $0 \rightarrow B_S \xrightarrow{i} E_U \xrightarrow{\pi} A_R \rightarrow 0$ of Rota-Baxter algebras. We also consider non-abelian $2$-cocycles as certain triple of maps $(\chi, \psi, \Phi)$ satisfying some identities. The set of all non-abelian $2$-cocycles modulo certain equivalence is called the non-abelian cohomology, denoted by $H^2_{nab}(A_R, B_S)$. We show that there is a one-to-one correspondence between equivalence classes of non-abelian extensions of $A_R$ by $B_S$ and the non-abelian cohomology $H^2_{nab}(A_R,B_S).$ (cf. Theorem \ref{non-ab-thm}).


\medskip

Given an extension of some algebraic structures, there is some interesting study about the inducibility of pair of automorphisms. Such study were first initiated by Wells in extensions of abstract groups \cite{wells}. In the same paper, the author also constructs a short exat sequence (popularly known as the fundamentalshort exact sequence of Wells) connecting various automorphism groups. A similar study has been developed in various other contexts \cite{hazra-habib,jin,passi}. In the present paper, we emphasize this study in non-abelian extensions of Rota-Baxter algebras. Let $0 \rightarrow B_S \xrightarrow{i} E_U \xrightarrow{\pi} A_R \rightarrow 0$ be a non-abelian extension of Rota-Baxter algebras. Let $\mathrm{Aut}_B (E_U)$ be the set of all Rota-Baxter automorphisms $\gamma \in \mathrm{Aut}(E_U)$ that satisfies $\gamma|_B \subset B$. Then there is a group homomorphism $\tau : \mathrm{Aut}_B (E_U) \rightarrow \mathrm{Aut}(B_S) \times \mathrm{Aut}(A_R)$, $\tau (\gamma) = (\gamma|_B, \pi \gamma s)$, where $s : A \rightarrow E$ is any section of the map $\pi$. We say that a pair $(\beta, \alpha) \in \mathrm{Aut}(B_S) \times \mathrm{Aut}(A_R)$ of Rota-Baxter automorphisms an inducible pair if $(\beta, \alpha)$ lies in the image of the map $\tau$. Given a non-abelian extension as above and a pair $(\beta, \alpha) \in \mathrm{Aut}(B_S) \times \mathrm{Aut}(A_R)$, we twist the non-abelian $2$-cocycle $(\chi, \psi, \Phi)$ to construct a new $2$-cocycle $( \chi_{(\beta, \alpha)}, \psi_{(\beta, \alpha)}, \Phi_{(\beta, \alpha)} )$. We show that a pair $(\beta, \alpha)$ is inducible if and only if the non-abelian $2$-cocycles $(\chi, \psi, \Phi)$ and $( \chi_{(\beta, \alpha)}, \psi_{(\beta, \alpha)}, \Phi_{(\beta, \alpha)} )$ are equivalent, thus they corresponds to the same element in $H^2_{nab}(A_R,B_S)$  (cf. Theorem \ref{thm-inducibility}).
To better understand the corresponding obstruction, we introduce the Wells map $\mathcal{W} : \mathrm{Aut}(B_S) \times \mathrm{Aut}(A_R) \rightarrow H^2_{nab}(A_R, B_S)$ associated to the given non-abelian extension of Rota-Baxter algebras. A pair $(\beta, \alpha)$ of Rota-Baxter automorphisms is inducible if and only if $\mathcal{W}((\beta, \alpha)) = 0$. Next, we obtain a short exact sequence (generalizing the fundamental sequence of Wells) connecting various automorphism groups and the non-abelian cohomology $H^2_{nab}(A_R, B_S)$  (cf. Theorem \ref{thm-wells}). We also construct two generalizations of the Wells exact sequence (cf. Theorem \ref{2thm-wells}). Finally, we find their role in extending a Rota-Baxter automorphism $\beta \in \mathrm{Aut}(B_S)$ and lifting a Rota-Baxter automorphism $\alpha \in \mathrm{Aut}(A_R)$ to an automorphism in $\mathrm{Aut}(E_U)$ (cf. Theorem \ref{2thm-wells-thm}).

\medskip

Abelian extensions of a Rota-Baxter algebra by a Rota-Baxter bimodule were classified in \cite{DasSKbimod}. We observe how this result can be obtained from our classification result for non-abelian extensions (cf. Theorem \ref{ab-ext-2ab}). We also discuss the inducibility problem and the fundamental sequence of Wells in the context of abelian extensions of a Rota-Baxter algebra (cf. Theorem \ref{ab-ext-ind-thm} and Theorem \ref{ab-ext-wells-thm}).

\medskip

In the last part of this paper, we focus on dendriform algebras. In \cite{das-dend,dialgebra,lod-val-book} the authors explicitly studied ordinary cohomology and abelian extensions for dendriform algebras. Here we consider non-abelian extensions of dendriform algebras and classify them by non-abelian cohomology (cf. Theorem \ref{non-ab-dend-thm}). This cohomology is much more technical as it involves certain combinatorial maps. Given an extension of dendriform algebras, we also find a necessary and sufficient condition for a pair of dendriform automorphisms to be inducible. We also discuss the fundamental sequence of Wells in the context of dendriform algebras.

\medskip

The paper is organized as follows. In Section \ref{sec-2}, we recall some preliminaries on Rota-Baxter algebras and dendriform algebras. Non-abelian extensions of Rota-Baxter algebras and their classifications are given in Section \ref{sec-3}. The inducibility of a pair of Rota-Baxter automorphisms is considered in Section \ref{sec-4}. We also construct the analogue of the fundamental sequence of Wells and its generalizations. In Section \ref{sec-5}, we show how the results of non-abelian extensions can be realized for abelian extensions of Rota-Baxter algebras. Finally, in Section \ref{sec-6}, we discuss extensions of dendriform algebras and the inducibility of pair of dendriform automorphisms.

\medskip

All modules, (multi)linear maps, associative algebras and tensor products are over a commutative ring ${\bf k}$ of characteristic $0$. The elements of the algebra $A$ are usually denoted by $a, b, c, \ldots$ and the algebra multiplication is denoted by $a \cdot_A b$, for $a, b \in A$. The elements of $B$ are usually denoted by $u, v, w, \ldots$ and the elements of $E$ are denoted by $e, f, \ldots$.

\section{Preliminaries on Rota-Baxter algebras and dendriform algebras}\label{sec-2}

In this section, we recall some basics of Rota-Baxter algebras and dendriform algebras. Our main references are \cite{aguiar,aguiar-pre-lie,dialgebra,guo-book}.

\medskip

An associative algebra is a {\bf k}-module $A$ equipped with a {\bf k}-bilinear product $\mu_A : A \times A \rightarrow A, ~(a,b) \mapsto a \cdot_A b$ satisfying the associativity:~ $(a \cdot_A b) \cdot_A c = a\cdot_A (b \cdot_A c)$, for all $a, b, c \in A$. An associative algebra as above may be simply denoted by $A$ when the product is clear from the context. For an associative algebra $A$, an $A$-bimodule is a {\bf k}-module $B$ equipped with {\bf k}-bilinear maps $l : A \times B \rightarrow B, ~(a,u) \mapsto a \cdot u$ and $r: B \times A \rightarrow B,~(u,a) \mapsto u \cdot a$ (called left and right $A$-actions, respectively) satisfying
\begin{align*}
(a \cdot_A b) \cdot u = a \cdot (b \cdot u), \quad (a \cdot u) \cdot b = a \cdot ( u \cdot b), \quad (u \cdot a) \cdot b = u \cdot (a \cdot_A b), ~\text{ for } a, b \in A, u \in B. 
\end{align*}

\begin{definition}
Let $A$ be an associative algebra. A {\bf Rota-Baxter operator} on $A$ is a ${\bf k}$-linear map $R: A \rightarrow A$ satisfying
\begin{align*}
R(a) \cdot_A R(b) = R \big(  R(a) \cdot_A b + a \cdot_A R(b) \big), \text{ for } a, b \in A.
\end{align*}
A {\bf Rota-Baxter algebra} is an associative algebra $A$ equipped with a distinguished Rota-Baxter operator $R$. We denote such a Rota-Baxter algebra simply by $A_R$.
\end{definition}

\begin{definition}
Let $A_R$ be a Rota-Baxter algebra. A {\bf Rota-Baxter bimodule} over $A_R$ consists of an $A$-bimodule $B$ together with a {\bf k}-linear map $S: B \rightarrow B$ satisfying
\begin{align*}
R(a) \cdot S(u) = S \big( R(a) \cdot u + a \cdot S(u)  \big)  ~ \text{ and } ~ S(u) \cdot R(a) = S \big(  S(u) \cdot a + u \cdot R(a) \big), \text{ for } a \in A, u \in B.
\end{align*}
\end{definition}

We denote a Rota-Baxter bimodule as above simply by $B_S$. The reader should not confuse with the notation of a Rota-Baxter algebra. It follows that any Rota-Baxter algebra $A_R$ is a Rota-Baxter bimodule over itself.

\medskip



\begin{definition}
A dendriform algebra is a triple $(\mathcal{A} , \prec_\mathcal{A}, \succ_\mathcal{A})$ consisting of a ${\bf k}$-module $\mathcal{A}$ equipped with two $\mathbf{k}$-bilinear operations $\prec_\mathcal{A}, \succ_\mathcal{A} : \mathcal{A} \otimes \mathcal{A} \rightarrow \mathcal{A}$ satisfying
\begin{align}
(a \prec_\mathcal{A} b) \prec_\mathcal{A} c = ~& a \prec_\mathcal{A} (b \prec_\mathcal{A} c + b \succ_\mathcal{A} c), \label{dend-1}\\
(a \succ_\mathcal{A} b) \prec_\mathcal{A} c =~& a \succ_\mathcal{A} (b \prec_\mathcal{A} c), \label{dend-2}\\
(a \prec_\mathcal{A} b + a \succ_\mathcal{A} b) \succ_\mathcal{A} c =~& a \succ_\mathcal{A} (b \succ_\mathcal{A} c), \text{ for } a, b, c \in \mathcal{A}. \label{dend-3}
\end{align}
A morphism between two dendriform algebras is given by a $\mathbf{k}$-linear map that preserves the structures.
\end{definition}

If $(\mathcal{A}, \prec_\mathcal{A}, \succ_\mathcal{A})$ is a dendriform algebra, it follows from (\ref{dend-1}), (\ref{dend-2}), (\ref{dend-3}) that the sum operation
\begin{align*}
a * b := a \prec_\mathcal{A} b + a \succ_\mathcal{A} b, \text{ for } a, b \in \mathcal{A}
\end{align*}
is associative. In other words, $(\mathcal{A}, *)$ is an associative algebra, called the `total' associative algebra. Thus, dendriform  algebras can be thought of as splitting of associative algebras.

Let $A_R$ be a Rota-Baxter algebra. Then the underlying ${\bf k}$-module $A$ inherits a dendriform algebra structure with the structure operations given by
\begin{align*}
a \prec b = a \cdot_A R(b) ~~~ \text{ and } ~~~ a \succ b = R(a) \cdot_A b, \text{ for } a, b \in A.
\end{align*}
This is called the dendriform algebra induced from the Rota-Baxter algebra $A_R.$ The corresponding total associative algebra is given by $(A, \ast_R)$, where $a \ast_R b = R(a) \cdot_A b + a \cdot_A R(b)$, for $a, b \in A$.

\section{Non-abelian extensions of Rota-Baxter algebras}\label{sec-3}
In this section, we study the non-abelian extension theory of Rota-Baxter algebras. We introduce the non-abelian cohomology group to parametrize equivalence classes on non-abelian extensions.

\begin{definition}\label{defin-nab}
(i) Let $A_R$ and $B_S$ be two Rota-Baxter algebras. A {\bf non-abelian extension} of $A_R$ by $B_S$ consists of a short exact sequence of Rota-Baxter algebras of the form
\begin{align}\label{nab-rb-seq}
\xymatrix{
0 \ar[r] & B_S \ar[r]^i & E_U \ar[r]^\pi & A_R \ar[r] & 0.
}
\end{align}
\end{definition}

We often denote a non-abelian extension as above simply by $E_U$ when the structure maps $i$ and $\pi$ are understood.

(ii) Two non-abelian extensions $E_U$ and $E'_{U'}$ are said to be {\bf equivalent} if there is a morphism $E_U \xrightarrow{\varphi} E'_{U'}$ of Rota-Baxter algebras making the following diagram commutative

\begin{align}\label{nab-rb-seq-eq}
\xymatrix{
0 \ar[r] & B_S \ar[r]^i \ar@{=}[d] & E_U \ar[r]^\pi \ar[d]^\varphi & A_R \ar[r] \ar@{=}[d] & 0 \\
0 \ar[r] & B_S \ar[r]_{i'} & E'_{U'} \ar[r]_{\pi'} & A_R \ar[r] & 0.
}
\end{align}

We denote by $\mathrm{Ext} (A_R, B_S)$ the set of all equivalence classes of non-abelian extensions of $A_R$ by $B_S$. In the following, we will parametrize the set $\mathrm{Ext} (A_R, B_S)$ by certain non-abelian cohomology. We first introduce the followings.

\begin{definition}\label{defn-nab-2-co-rb}
(i) Let $A_R$ and $B_S$ be two Rota-Baxter algebras. A {\bf non-abelian $2$-cocycle} on $A_R$ with values in $B_S$ consists of a triple $(\chi, \psi, \Phi)$ of maps 
\begin{align*}
\chi : A^{\otimes 2} \rightarrow B,  \quad \psi : (A \otimes B) \oplus (B \otimes A) \rightarrow B ~~~~ ~~~ \text{ and } ~~~~ ~~~ \Phi : A \rightarrow B ~~~ \text{ satisfying for } a, b \in A, u \in B,
\end{align*}
\begin{align}
&\begin{cases}\label{nabb1}
\psi (a, \psi (b, u))= \psi (a \cdot_A b, u) + \chi (a, b) \cdot_B u,\\
\psi (a, \psi (u, b)) = \psi (\psi (a, u), b), \\
\psi (\psi (u, a), b) = \psi (u, a \cdot_A b) + u \cdot_B \chi (a, b),\\
\end{cases}\\
&  \quad \psi (a , \chi (b, c)) - \chi (a \cdot_A b , c) + \chi (a, b \cdot_A c) - \psi (\chi (a, b), c) = 0, \label{nabb2}\\
&\begin{cases}\label{nabb3}
\psi (R(a), S(u)) = S \big( \psi (R(a), u) + \psi (a, S(u)) \big) + S (\Phi (a) \cdot_B u) - \Phi (a) \cdot_B S(u), \\
\psi (S(u), R(a)) = S \big( \psi (S(u), a) + \psi (u, R(a)) \big) + S (u \cdot_B \Phi (a)) - S(u) \cdot_B \Phi(a), 
\end{cases}\\ \medskip
&\quad \chi (R(a), R(b)) - S \big(  \chi (R(a), b) + \chi (a, R(b)) \big) - \Phi \big( R(a) \cdot_A b + a \cdot_A R(b) \big) \label{nabb4}\\
&\qquad + \psi (R(a), \Phi (b)) + \psi (\Phi (a), R(b)) - S \big( \psi (\Phi (a), b) + \psi (a, \Phi (b))   \big) + \Phi (a) \cdot_B \Phi (b) = 0. \nonumber
\end{align}

(ii) Two non-abelian $2$-cocycles $(\chi, \psi, \Phi)$ and $(\chi', \psi', \Phi')$ are said to {\bf equivalent} if there exists a linear map $\phi : A \rightarrow B$ satisfying
\begin{align*}
&\begin{cases}
\psi (a, u) - \psi' (a, u) = \phi (a) \cdot_B u,\\
\psi (u, a) - \psi' (u, a) = u \cdot_B \phi (a),
\end{cases}\\
& \quad \chi (a, b) - \chi' (a, b) = \psi' (a, \phi (b)) + \psi' (\phi (a) , b) - \phi (a \cdot_A b) + \phi (a) \cdot_B \phi (b),\\
& \quad \Phi (a) - \Phi'(a) = S \phi (a) - \phi R(a), \text{ for } a, b \in A, {u \in B}.
\end{align*}
\end{definition}

In this case, we simply write $(\chi, \psi, \Phi) \overset{\phi}{\sim} (\chi', \psi', \Phi')$. The set of all equivalence classes of non-abelian $2$-cocycles are denoted by $H^2_{nab}(A_R, B_S)$. It is also called the {\bf non-abelian cohomology}.

\begin{theorem}\label{non-ab-thm}
Let $A_R$ and $B_S$ be two Rota-Baxter algebras. Then there is a one-to-one correspondence between the equivalence classes of non-abelian extensions of $A_R$ by $B_S$ and the non-abelian cohomology $H^2_{nab}(A_R, B_S)$. In other words,
\begin{align*}
\mathrm{Ext} (A_R, B_S) \cong H^2_{nab}(A_R, B_S).
\end{align*}
\end{theorem}

\begin{proof}
Let $E_U$ be a non-abelian extension of $A_R$ by $B_S$ as of (\ref{nab-rb-seq}).  Let $s$ be a section of the map $\pi$ (i.e. $s : A \rightarrow E$ is a linear map satisfying $\pi s = \mathrm{id}_A$). Depending on that section, we define a triple $(\chi^s, \psi^s, \Phi^s)$ of maps as
\begin{align*}
&\chi^s : A^{\otimes 2} \rightarrow B, ~~~ \chi^s (a, b) := s(a) \cdot_E s(b) - s(a \cdot_A b), \\
&\psi^s : (A \otimes B) \oplus (B \otimes A) \rightarrow B, ~~~ \begin{cases}
\psi^s (a, u) := s(a) \cdot_E u, \\
\psi^s (u, a) := u \cdot_E s(a),\\
\end{cases}\\
&\Phi^s : A \rightarrow B, ~~~ \Phi^s (a) := (Us- sR)(a),
\end{align*}
for $a, b \in A$ and $u \in B$. It has been observed in \cite{gouray} that the maps $\chi^s$ and $\psi^s$ satisfy the identities (\ref{nabb1}) and (\ref{nabb2}). In the following, we will show that $(\chi^s, \psi^s, \Phi^s)$ satisfy the identities (\ref{nabb3}) and (\ref{nabb4}). To see these, we first observe that
\begin{align*}
&\psi^s (R(a), S(u)) - S \big(  \psi^s (R(a), u) + \psi^s (a, S(u)) \big) - S (\Phi^s (a) \cdot_B u) + \Phi^s (a) \cdot_B S(u) \\
&= \cancel{sR (a) \cdot_E S(u)} - S \big( sR(a) \cdot_E u + s(a) \cdot_E S(u)  \big) - S \big(  Us(a) \cdot_E u - sR(a) \cdot_E u \big) \\
&\qquad \qquad + \big(  Us(a) \cdot_E S(u) - \cancel{sR(a) \cdot_E S(u)} \big) \\
&= - U \big(  \cancel{s R(a) \cdot_E u} + s(a) \cdot_E U(u)   \big)  - U \big( Us(a) \cdot_E u - \cancel{sR(a) \cdot_E u} \big) + Us(a) \cdot_E U(u) \quad (\text{as } S = U|_B) \\
&=  - U \big(  s(a) \cdot_E U(u) + Us(a) \cdot_E u \big)  + Us (a) \cdot_E U(u) = 0 \quad (\text{as } U \text{ is a Rota-Baxter operator})
\end{align*}
and
\begin{align*}
&\psi^s (S(u), R(a)) - S \big(  \psi^s (S(u),a) + \psi^s (u, R(a)) \big) - S (u \cdot_B \Phi^s (a)) + S(u) \cdot_B \Phi^s (a) \\
&= \cancel{S (u) \cdot_E sR(a)} - S \big( S(u) \cdot_E s(a) + u \cdot_E sR(a)  \big) - S \big( u \cdot_E Us (a) - u \cdot_E sR(a) \big) \\
& \qquad \qquad + S (u) \cdot_E Us(a) - \cancel{S(u) \cdot_E sR (a)} \\
&= - U \big( U(u) \cdot_E s(a) + \cancel{u \cdot_E sR(a)}  \big) - U \big( u \cdot_E Us(a) - \cancel{u \cdot_E sR(a)} \big) + U(u) \cdot_E Us(a) \quad (\text{as } S = U|_B) \\
&= - U \big( U(u) \cdot_E s(a) + u \cdot_E Us(a) \big) + U(u) \cdot_E Us(a) = 0 \quad (\text{as } U \text{ is a Rota-Baxter operator}).
\end{align*}
Hence the identity (\ref{nabb3}) follows. Finally, we have
\begin{align*}
&\chi^s (R(a), R(b)) - S \big(  \chi^s (R(a), b) + \chi^s (a, R(b)) \big) - \Phi^s \big( R(a) \cdot_A b + a \cdot_A R(b) \big) \\
&\qquad + \psi^s (R(a), \Phi^s (b)) + \psi^s (\Phi^s (a), R(b)) - S \big( \psi^s (\Phi^s (a), b) + \psi^s (a, \Phi^s (b))   \big) + \Phi^s (a) \cdot_B \Phi^s (b) \\
&= \cancel{sR(a) \cdot_E sR(b)} \underbrace{- s \big(  R(a) \cdot_A R(b)}_{(1)} \big) - S \big( \underbrace{sR(a) \cdot_E s(b)}_{(2)} \underbrace{- s (R(a) \cdot_A b) }_{(3)}  \big) - S \big( \underbrace{s(a) \cdot_E sR(b)}_{(5)} \underbrace{- s (a \cdot_A R(b) ) }_{(6)}  \big) \\
& \underbrace{- Us (R(a) \cdot_A b)}_{(3)} \underbrace{+ sR ( R(a) \cdot_A b)}_{(1)} \underbrace{- Us (a \cdot_A R(a))}_{(6)} \underbrace{+ sR (a \cdot_A R(b))}_{(1)} + \underbrace{sR(a) \cdot_E Us (b)}_{(7)} - \cancel{sR(a) \cdot_E sR(b)} \\
&\underbrace{+ Us (a) \cdot_E sR (b)}_{(8)} - \cancel{sR(a) \cdot_E sR(b)} - S \big( Us(a) \cdot_E s(b) \underbrace{- sR(a) \cdot_E s(b)}_{(2)} + s(a) \cdot_E Us(b) \underbrace{- s(a) \cdot_E sR (b)}_{(5)}   \big) \\
& + Us(a) \cdot_E Us(b) \underbrace{- Us(a) \cdot_E sR(b)}_{(8)} \underbrace{-sR(a) \cdot_E Us (b)}_{(7)} + \cancel{sR(a) \cdot_E sR(b)} \\
&= - U \big( Us(a) \cdot_E s(b) + s(a) \cdot_E Us(b)  \big) + Us(a) \cdot_E Us (b)
\end{align*}
which vanishes as $U$ is a Rota-Baxter operator. This verifies the identity (\ref{nabb4}). Hence $(\chi^s, \psi^s, \Phi^s)$ is a non-abelian $2$-cocycle.\\

If $t$ is another section of $\pi$, and $(\chi^t, \psi^t, \Phi^t)$ is the corresponding non-abelian $2$-cocycle, then it is easy to verify that 
\begin{align*}
(\chi^s, \psi^s, \Phi^s) \overset{\phi}{\sim} (\chi^t, \psi^t, \Phi^t), \text{ where } \phi := s - t : A \rightarrow B.
\end{align*}
To see this, we observe that
\begin{align*}
\begin{cases}
\psi^s (a, u) - \psi^t (a, u) = s(a) \cdot_E u - t(a) \cdot_E u = \phi (a) \cdot_B u,\\
\psi^s (u, a) - \psi^t (u, a) = u \cdot_E s(a) - u \cdot_E t(a) = u \cdot_B \phi(a),
\end{cases}
\end{align*}
\begin{align*}
\chi^s (a,b) - \chi^t (a, b) =~& s(a) \cdot_E s(b) - s (a \cdot_A b) - t(a) \cdot_E t(b) + t (a \cdot_A b) \\
=~& (\phi + t)(a) \cdot_E  (\phi + t)(b) -  (\phi + t) (a \cdot_A b) - t(a) \cdot_E t(b) + t(a \cdot_A b) \\
=~& t(a) \cdot_E \phi(b) + \phi(a) \cdot_E t(b) - \phi (a \cdot_A b) + \phi(a) \cdot_B \phi(b)  \\
=~& \psi^t (a, \phi (b)) + \psi^t (\phi (a), b) - \phi (a \cdot_A b) + \phi (a) \cdot_B \phi (b),
\end{align*}
\begin{align*}
\Phi^s (a) - \Phi^t (a) =~& (Us - sR)(a) - (Ut - tR)(a) \\
=~& U \phi (a) - \phi R (a) = S \phi (a) - \phi R(a).
\end{align*}
Therefore, $(\chi^s, \psi^s, \Phi^s)$ and $(\chi^t, \psi^t, \Phi^t)$ corresponds to the same element in $H^2_{nab} (A_R, B_S)$.

\medskip

Next, let $E_U$ and $E'_{U'}$ be two equivalent non-abelian extensions as of (\ref{nab-rb-seq-eq}). For any section $s$ of the map $\pi$, we have $\pi' (\varphi s) = (\pi' \varphi) s = \pi s = \mathrm{id}_A$. This shows that $s' := \varphi s$ is a section of the map $\pi'$. If $(\chi^{s'}, \psi^{s'}, \Phi^{s'})$ is the corresponding non-abelian $2$-cocycle, then for any $a, b \in A$ and $u \in B$,
\begin{align*}
\chi^{s'} (a, b) = s' (a) \cdot_{E'} s'(b) - s'(a \cdot_A b) =~& \varphi s(a) \cdot_{E'} \varphi s(b) - \varphi s (a \cdot_A b) \\
=~& \varphi \big( s(a) \cdot_E s(b) - s (a \cdot_A b) \big) \\
=~& \varphi (\chi^s (a, b)) = \chi^s (a,b),
\end{align*}
\begin{align*}
\begin{cases}
\psi^{s'} (a, u) = s'(a) \cdot_{E'} u = \varphi s (a) \cdot_{E'} u = \varphi (s(a) \cdot_E u) = \varphi (\psi^s (a, u)) = \psi^s (a, u),\\
\psi^{s'} (u, a) =  u \cdot_{E'} s'(a) = u \cdot_{E'} \varphi s (a) = \varphi (u \cdot_E s(a))= \varphi (\psi^s (u,a)) = \psi^s (u, a),
\end{cases}
\end{align*}
\begin{align*}
\Phi^{s'} (a) = (U's' - s' R)(a) =~& U' (\varphi s (a)) - \varphi s (R(a)) \\
=~& \varphi (Us - sR)(a) = \varphi (\Phi^s (a)) = \Phi^s (a).
\end{align*}
Therefore, we have $(\chi^{s'}, \psi^{s'}, \Phi^{s'}) = (\chi^s, \psi^s, \Phi^s)$. Hence we obtain a map
\begin{align*}
\Lambda : \mathrm{Ext} (A_R, B_S) \rightarrow H^2_{nab} (A_R, B_S), ~ [E_U] \mapsto [(\chi^s, \psi^s, \Phi^s)] \text{ for any section } s \text{ of the map } \pi.\\
\end{align*}

To obtain a map in the other direction, we first start with a non-abelian $2$-cocycle $(\chi, \psi, \Phi)$. Consider the ${\bf k}$-module $A \oplus B$ with the bilinear product
\begin{align*}
(a, u) \cdot_{\chi, \psi} (b, v) := (a \cdot_A b,~ \psi(a, v) + \psi (u, b) + \chi (a, b) + u \cdot_B v),
\end{align*}
for $(a, u), (b, v) \in A \oplus B$. Since $\chi, \psi$ satisfy the identities (\ref{nabb1}) and (\ref{nabb2}), it follows that the product $\cdot_{\chi, \psi}$ is associative, which turns $A \oplus B$ into an associative algebra (denoted by $A \oplus_{\chi, \psi} B$). Further, we define a linear map $U_\Phi : A \oplus B \rightarrow  A \oplus B$ by $$U_\Phi (a, u) = (R(a), S(u) + \Phi (a)),$$
for $(a, u) \in A \oplus B$. Then we have
\begin{align*}
&U_\Phi (a, u) \cdot_{\chi, \psi} U_\Phi (b, v) \\
&= \big( R(a) , S(u) + \Phi (a) \big) \cdot_{\chi, \psi} \big( R(b) , S(v) + \Phi (b) \big) \\
&= \big( R(a) \cdot_A R(b), ~ \underbrace{\psi (R(a), S(v))}_{(A1)} + \underbrace{\psi (R(a), \Phi (b))}_{(B1)} + \underbrace{\psi (S(u), R(b))}_{(C1)} + \underbrace{\psi (\Phi (a), R(b))}_{(B2)} + \underbrace{\chi (R(a), R(b))}_{(B3)} \\
& ~~ + \underbrace{S(u) \cdot_B S(v)}_{(D)} + \underbrace{S(u) \cdot_B \Phi (b) }_{(C2)} + \underbrace{\Phi (a) \cdot_B S(v)}_{(A2)} + \underbrace{\Phi (a) \cdot_B \Phi (b)}_{(B4)}   \big) \\
&= \big(  R (R(a) \cdot_A b) + R (a \cdot_A R(b)), ~ \underbrace{ S ( \psi ( R(a), v ) + \psi (a, S(v)) ) + S (\Phi(a) \cdot_B v)}_{(A1) + (A2)}   \\
& ~~ + \underbrace{ S ( \psi (\Phi (a), b)) + \psi (a, \Phi (b)) ) + S ( \chi (R(a), b) + \chi (a, R(b)) ) + \Phi ( R(a) \cdot_A b + a \cdot_A R(b) )  }_{(B1) + (B2) + (B3) + (B4)}   \\
& ~~ + \underbrace{ S ( \psi (S(u), b)  + \psi (u, R(b)) ) + S (u \cdot_B \Phi (b)) }_{(C1) + (C2)} +  \underbrace{S (  S(u) \cdot_B v + u \cdot_B S(v)  )}_{(D)} \big) \\
&= \big( R (R(a) \cdot_A b), ~ S \big(  \psi (R(a), v) + \psi (S(u), b) + \psi (\Phi(a), b) + \chi (R(a), b) +   (S(u) + \Phi(a)) \cdot_B v \big) + \Phi (R(a) \cdot_A b)  \big) \\
& ~~ + \big(  R (a \cdot_A R(b)), ~ S \big(  \psi (a, S(v)) + \psi (a, \Phi (b)) + \psi (u, R(b)) + \chi (a, R(b)) + u \cdot_B (S(u) + \Phi (b))  \big) + \Phi (a \cdot_A R(b))  \big) \\
&= U_\Phi \big(  R(a) \cdot_A b, ~  \psi (R(a), v) + \psi (S(u), b) + \psi (\Phi(a), b) + \chi (R(a), b) +   (S(u) + \Phi(a)) \cdot_B v  \big)  \\
& ~~ + U_\Phi \big( a \cdot_A R(b), ~  \psi (a, S(v)) + \psi (a, \Phi (b)) + \psi (u, R(b)) + \chi (a, R(b)) + u \cdot_B (S(u) + \Phi (b))  \big) \\
&= U_\Phi \big(  (R(a) , S(u) + \Phi (a)) \cdot_{\chi, \psi} (b, v) ~+~ (a, u) \cdot_{\chi,\psi} (R(b) , S(v) + \Phi (b))  \big) \\
&= U_\Phi \big(  U_\Phi (a,u) \cdot_{\chi, \psi} (b, v) ~+~ (a, u) \cdot_{\chi, \psi} U_\Phi (b, v) \big).
\end{align*}
This shows that $U_\Phi$ is a Rota-Baxter operator on the algebra $A \oplus_{\chi, \psi} B$. In other words, $ (A \oplus_{\chi, \psi} B)_{U_\Phi}$ is a Rota-Baxter algebra. It is then easy to see that
\[
\xymatrix{
0 \ar[r] & B_S \ar[r]^i_{b \mapsto (0,b)} & (A \oplus_{\chi, \psi} B)_{U_\Phi} \ar[r]^\pi_{(a,b) \mapsto a} & A_R \ar[r] & 0
}
\]
is a non-abelian extension of Rota-Baxter algebras. Finally, if $(\chi, \psi, \Phi)$ and $(\chi', \psi', \Phi')$ are two equivalent non-abelian $2$-cocycles (see Definition \ref{defn-nab-2-co-rb}(ii)), then we can define a map $\Theta : A \oplus B \rightarrow A \oplus B$ by $\Theta ((a, u)) = (a, u + \phi (a))$, for $(a, u) \in A \oplus B$. For any $(a, u), (b, v) \in A \oplus, B$, we have
\begin{align*}
&\Theta ((a,u) \cdot_{\chi, \psi} (b, v)) \\
&= \Theta \big(  a \cdot_A b, ~ \psi (a,v) + \psi (u,b) + \chi (a,b) + u \cdot_B v \big) \\
&= \big( a \cdot_A b, ~\psi (a, v) + \psi (u,b) + \chi (a,b) + u \cdot_B v + \phi (a \cdot_A b)   \big) \\
&= \big(   a \cdot_A b, ~ \psi' (a, v) + \phi(a) \cdot_B v + \psi' (u,b) + u \cdot_B \phi (b) \\
& \qquad + \chi' (a, b) + \psi (\phi (a), b) + \psi (a, \phi(b)) - \cancel{\phi (a \cdot_A b)} - \phi (a) \cdot_A \phi(b) + u \cdot_B v + \cancel{\phi (a \cdot_A b)} \big) \\
&= \big( a \cdot_A b,~ \psi' (a, v) + \phi (a) \cdot_B v + \psi' (u, b) + u \cdot_B \phi (b)  \\
& \qquad + \chi' (a,b) + \psi' (\phi (a), b) + \phi (a) \cdot_B \phi (b)  + \psi' (a, \phi (b)) + \cancel{\phi (a) \cdot_B \phi(b)} - \cancel{\phi(a) \cdot_B \phi(b)} + u \cdot_B v \big) \\
&= (a , u + \phi (a)) \cdot_{\chi', \psi'} (b, v + \phi (b)) \\
&= \Theta (a, u) \cdot_{\chi', \psi'} \Theta (b, v)
\end{align*}
and
\begin{align*}
(U_{\Phi'} \circ \Theta)(a, u) =~& U_{\Phi'} (a, u+ \phi (a)) \\
=~& \big(  R(a), S(u) + S \phi (a) + \Phi' (a) \big) \\
=~& \big(  R(a), S(u) + \Phi (a) + \phi R (a) \big) \\
=~& \Theta \big( R(a), S(u) + \Phi (a)  \big) = (\Theta \circ U_\Phi)(a, u).
\end{align*}
This shows that $(A \oplus_{\chi, \psi} B)_{U_\Phi} \xrightarrow{\Theta} (A \oplus_{\chi', \psi'} B)_{U_{\Phi'}}$ is a morphism of Rota-Baxter algebras. The map $\Theta$ is infact an equivalence between non-abelian extensions. As a summary, we obtain a map
\begin{align*}
\Upsilon : H^2_{nab}(A_R, B_S) \rightarrow \mathrm{Ext} (A_R, B_S), ~ [(\chi, \psi, \Phi)] \mapsto [(A \oplus_{\chi, \psi} B)_{U_\Phi}].
\end{align*}
Finally, it is tedious but straightforward to see that the maps $\Lambda$ and $\Upsilon$ are inverses to each other. This completes the proof.
\end{proof}

\section{The inducibility problem and the fundamental exact sequence of Wells}\label{sec-4}
Let $A_R$ and $B_S$ be two Rota-Baxter algebras. A {morphism} $A_R \xrightarrow{\varphi} B_S$ of Rota-Baxter algebras is an algebra homomorphism $\varphi : A \rightarrow B$ satisfying $S \circ \varphi = \varphi \circ R$. It is said to be an isomorphism if $\varphi$ is a linear isomorphism. Let $A_R$ be a fixed Rota-Baxter algebra. The set of all automorphisms of the Rota-Baxter algebra $A_R$ forms a group $\mathrm{Aut} (A_R)$, called the automorphism group. Given a non-abelian extension $0 \rightarrow B_S \xrightarrow{i} E_U \xrightarrow{\pi} A_R \rightarrow 0$ of Rota-Baxter algebras, here we study the inducibility of a pair of Rota-Baxter automorphisms $(\beta, \alpha) \in \mathrm{Aut}(B_S) \times \mathrm{Aut} (A_R)$ from an automorphism in $\mathrm{Aut}(E_U)$. To better understand the corresponding obstruction, we introduce the analogue of the fundamental exact sequence of Wells. We also construct two generalizations of this sequence. In particular, we answers when a Rota-Baxter automorphism $\beta \in \mathrm{Aut} (B_S)$ extends and a Rota-Baxter automorphism $\alpha \in \mathrm{Aut}(A_R)$ lifts to a Rota-Baxter automorphism in $\mathrm{Aut}(E_U)$.

\medskip

Let $0 \rightarrow B_S \xrightarrow{i} E_U \xrightarrow{\pi} A_R \rightarrow 0$ be a non-abelian extension of Rota-Baxter algebras. Let 
\begin{align*}
\mathrm{Aut}_B (E_U) = \{ \gamma \in \mathrm{Aut}(E_U) ~|~ \gamma|_B \subset B \}.
\end{align*}
It follows that, if $\gamma \in \mathrm{Aut}_B (E_U)$ then $\gamma|_B \in \mathrm{Aut}(B_S)$. Moreover, for any $\gamma \in \mathrm{Aut}_B (E_U)$, we may define a map $\overline{\gamma} : A \rightarrow A$ by
\begin{align*}
\overline{\gamma} (a) := \pi \gamma s (a), \text{ for } a \in A.
\end{align*}
Here $s$ is any section of the map $\pi$. Note that the map $\overline{\gamma}$ is independent of the choice of $s$. Since $E$ is isomorphic to the module $A \oplus B$, the map $\pi$ can be regarded as the projection onto the submodule $A$. As $\gamma$ is an automorphism on $E$ preserving the submodule $B$, it also preserves $A$. Thus, $\overline{\gamma} = \pi \gamma s$ is a bijection on $A$. Moreover, for any $a, b \in A$,
\begin{align*}
\overline{\gamma} (a \cdot_A b) = \pi \gamma (s (a \cdot_A b)) =~& \pi \gamma \big(  s(a) \cdot_E s(b) - \chi (a, b) \big) \\
=~& \pi \gamma \big( s(a) \cdot_E s(b) \big) \quad (\text{as } \gamma|_B \subset B \text{ and } \pi|_B = 0) \\
=~& \pi \gamma s (a) \cdot_A \pi \gamma s (b) = \overline{\gamma} (a) \cdot_A \overline{\gamma} (b)
\end{align*}
and
\begin{align*}
(R \circ \overline{\gamma} - \overline{\gamma} \circ R) (a) =~& (R \pi \gamma s - \pi \gamma s R) (a) \\
=~& (\pi U \gamma s - \pi \gamma s R) (a) \quad (\text{as } R \pi = \pi U) \\
=~& \pi \gamma (Us - sR)(a)  \quad (\text{as } U \gamma = \gamma U) \\
=~& 0 \quad (\text{as } (Us - sR)(a) \in B, ~ \gamma|_B \subset B \text{ and } \pi|_B = 0).
\end{align*}
This shows that $\overline{\gamma} \in \mathrm{Aut}(A_R).$ Hence we obtain a map
\begin{align*}
\tau : \mathrm{Aut}_B (E_U) \rightarrow \mathrm{Aut}(B_S) \times \mathrm{Aut}(A_R), ~ \tau (u) := (\gamma|_B , \overline{\gamma})
\end{align*}
which is a group homomorphism. A pair $(\beta, \alpha) \in \mathrm{Aut}(B_S) \times \mathrm{Aut}(A_R)$ of Rota-Baxter automorphisms is said to be inducible if it lies in the image of $\tau$.

In the following, we will find a necessary and sufficient condition for the inducibility of a pair of Rota-Baxter automorphism $(\beta, \alpha) \in \mathrm{Aut}(B_S) \times \mathrm{Aut}(A_R)$. Before state our main result, we need some more notations. Let $(\chi, \psi, \Phi)$ be the non-abelian $2$-cocycle associated to the non-abelian extension $0 \rightarrow B_S \xrightarrow{i} E_U \xrightarrow{\pi} A_R \rightarrow 0$ and a fixed section $s$. We define a new triple $(\chi_{(\beta, \alpha)}, \psi_{(\beta, \alpha)}, \Phi_{(\beta, \alpha)})$ of maps
\begin{align*}
\chi_{(\beta, \alpha)} \in \mathrm{Hom}(A^{\otimes 2}, B), \quad \psi_{(\beta, \alpha)} : (A \otimes B) \oplus (B \otimes A) \rightarrow B   ~~~ \text{ and } ~~~ \Phi_{(\beta, \alpha)} \in \mathrm{Hom} (A, B)  ~~\text{ by }
\end{align*}
\begin{align*}
\chi_{(\beta, \alpha)} (a, b) :=~& \beta \circ \chi (\alpha^{-1} (a), \alpha^{-1}(b)), \\
\psi_{(\beta, \alpha)} (a, u) :=~& \beta \circ \psi (\alpha^{-1}(a), \beta^{-1}(u)), \\
\beta_{(\beta, \alpha)} (u, a) :=~& \beta \circ \psi (\beta^{-1}(u), \alpha^{-1}(a)), \\
\Phi_{(\beta, \alpha)} (a) :=~& \beta \circ \Phi (\alpha^{-1} (a)),
\end{align*}
for $a, b \in A$ and $u \in B$. Note that $(\chi, \psi, \Phi)$ is a non-abelian $2$-cocycle implies that the identities (\ref{nabb1}), (\ref{nabb2}), (\ref{nabb3}) and (\ref{nabb4}) are hold. In these identities, if we replace $a, b, u$ by $\alpha^{-1}(a), \alpha^{-1} (b), \beta^{-1}(u)$ respectively, we obtain the corresponding identities for the triple $(\chi_{(\beta, \alpha)}, \psi_{(\beta, \alpha)}, \Phi_{(\beta, \alpha)})$. This shows that $(\chi_{(\beta, \alpha)}, \psi_{(\beta, \alpha)}, \Phi_{(\beta, \alpha)})$ is a non-abelian $2$-cocycle.

\begin{theorem}\label{thm-inducibility}
Let $0 \rightarrow B_S \xrightarrow{i} E_U \xrightarrow{\pi} A_R \rightarrow 0$ be a non-abelian extension of Rota-Baxter algebras. A pair $(\beta, \alpha) \in \mathrm{Aut}(B_S) \times \mathrm{Aut} (A_R)$ of Rota-Baxter automorphisms is inducible if and only if the non-abelian $2$-cocycles $(\chi_{(\beta, \alpha)}, \psi_{(\beta, \alpha)}, \Phi_{(\beta, \alpha)})$ and $(\chi, \psi, \Phi)$ are equivalent (i.e. they corresponds to the same element in $H^2_{nab} (A_R, B_S)$).
\end{theorem}

\begin{proof}
Suppose $(\beta, \alpha)$ is an inducible pair of Rota-Baxter automorphisms. Thus, there exists an element $\gamma \in \mathrm{Aut}_B (E_U)$ such that $\gamma|_B = \beta$ and $\pi \gamma s = \alpha$. For any $a \in A$, we observe that 
\begin{align*}
\pi \big( (\gamma s \alpha^{-1} - s)(a)  \big) = a - a = 0 \quad (\text{as } \pi s = \mathrm{id}_A).
\end{align*}
This shows that $(\gamma s \alpha^{-1} - s)(a) \in \mathrm{ker} (\pi) = \mathrm{im}(i) \cong B$. We define a map $\phi : A \rightarrow B$ by $\phi (a) := (\gamma s \alpha^{-1} -s )(a)$, for $a \in A$. Then for any $a \in A$ and $u \in B$, we have
\begin{align*}
\psi_{(\beta, \alpha)} (a, u) - \psi (a, u) =~& \beta \psi (\alpha^{-1}(a), \beta^{-1}(u)) - \psi (a, u) \\
=~& \beta \big( s\alpha^{-1} (a) \cdot_E \beta^{-1} (u) \big) - \big(  s(a) \cdot_E u \big) \\
=~& \big(  \gamma s \alpha^{-1} (a) \cdot_E u \big) - \big(  s(a) \cdot_E u \big) \\
=~& \big( \gamma s \alpha^{-1}(a) - s(a) \big) \cdot_B u = \phi (a) \cdot_B u.
\end{align*}
Similarly, $\psi_{(\beta, \alpha)} (u, a) - \psi (u, a) = u \cdot \phi (a).$ Moreover, for any $a, b \in A$,
\begin{align*}
&\chi_{(\beta, \alpha)} (a, b) - \chi (a, b) \\
&= \beta \chi \big( \alpha^{-1}(a), \alpha^{-1}(b) \big) - \chi (a, b)  \\
&= \beta \big(  s \alpha^{-1} (a) \cdot_E s \alpha^{-1}(b) - s (\alpha^{-1}(a) \cdot_A \alpha^{-1}(b))   \big) - \big( s(a) \cdot_E s(b) - s (a \cdot_A b)   \big) \\
&= \beta \big(  s \alpha^{-1}(a) \cdot_E \beta^{-1} (\gamma s \alpha^{-1} - s)(b)  \big) + \beta \big( \beta^{-1} (\gamma s \alpha^{-1} - s)(a) \cdot_E s \alpha^{-1}(b)  \big) - \gamma s \alpha^{-1} (a \cdot_A b) + s (a \cdot_A b) \\
& \quad - \gamma s \alpha^{-1} (a) \cdot_E \gamma s \alpha^{-1} (b) + \gamma s \alpha^{-1} (a) \cdot_E s (b) + s (a) \cdot_E \gamma s \alpha^{-1} (b) - s(a) \cdot_E s(b) \\
& \qquad \qquad \qquad  \qquad \qquad \qquad (\text{by adding and subtracting some terms}) \\
&= \beta \psi (\alpha^{-1}(a), \beta^{-1} \phi (b)) + \beta \psi (\beta^{-1} \phi (a), \alpha^{-1}(b)) - (\gamma s \alpha^{-1} - s) (a \cdot_A b) - (\gamma s \alpha^{-1} - s)(a) \cdot_B (\gamma s \alpha^{-1} - s)(b) \\
&= \psi_{(\beta, \alpha)} (a, \phi(b)) + \psi_{(\beta, \alpha)} (\phi (a), b) - \phi (a \cdot_A b) - \phi (a) \cdot_B \phi (b). 
\end{align*}
Finally, 
\begin{align*}
\Phi_{(\beta, \alpha)} (a) - \Phi (a) =~& \beta \Phi (\alpha^{-1} (a)) - \Phi (a) \\
=~& \beta (Us -s R) \alpha^{-1} (a) - (Us -sR)(a) \\
=~& \gamma (Us -s R) \alpha^{-1} (a) - (Us -sR)(a) \\
=~& (U \gamma s - \gamma s R) \alpha^{-1} (a) - (Us -sR)(a) \\
=~& U (\gamma s \alpha^{-1} - s)(a) - (\gamma s \alpha^{-1} - s) R(a) \\
=~& S (\gamma s \alpha^{-1} - s)(a) - (\gamma s \alpha^{-1} - s) R(a) = (S \phi - \phi R)(a).
\end{align*}
Thus, it follows from Definition \ref{defin-nab}(ii) that $(\chi_{(\beta, \alpha)}, \psi_{(\beta, \alpha)}, \Phi_{(\beta, \alpha)}) \overset{\phi}{\sim} (\chi, \psi, \Phi)$.

\medskip

Conversely, suppose that the non-abelian $2$-cocycles $(\chi_{(\beta, \alpha)}, \psi_{(\beta, \alpha)}, \Phi_{(\beta, \alpha)})$ and $(\chi, \psi, \Phi)$ are equivalent, and the equivalence is given by a map $\phi : A \rightarrow B$. Note that $s$ is a section of the map $\pi$ implies that $E \cong A \oplus B$ as a {\bf k}-module. Thus, an element $e \in E$ can be uniquely written as $e = (a, u)$, for some $a \in A$ and $u \in B$. We now define a map $\gamma : E \rightarrow E$ by
\begin{align*}
\gamma (e) = (\alpha (a), \beta (u) + \phi \alpha (a)), \text{ for } e = (a,u). 
\end{align*}
If for any $e = (a, u)$, the image $\gamma (e) = (\alpha (a), \beta (u) + \phi \alpha (a)) = 0$ then it follows that $\alpha (a) = 0$ (which implies that $a=0$) and $\beta (u) + \phi \alpha (a) = 0$. Thus, we have $\beta (u) = 0$ which implies that $u = 0$. Therefore, $e = (a, u) = 0$. This shows that $\gamma$ is injective. The map $\gamma$ is also surjective, as for any $e = (a, u) \in E$, we consider the element $e' = \big( \alpha^{-1} (a) , \beta^{-1} (u) - \beta^{-1} \phi (a) \big) \in E$ and observe that
\begin{align*}
\gamma (e' ) = \gamma \big( \alpha^{-1} (a) , \beta^{-1} (u) - \beta^{-1} \phi (a) \big) = (a, u) = e.
\end{align*}
Next, for any $e = (a,u)$ and $f = (b,v)$ of elements in $E$, we have
\begin{align*}
 \gamma (e) \cdot_E \gamma (f) =~& \gamma (a,u) \cdot_E \gamma (b, v) \\
 =~& \big(  \alpha (a), \beta(u) + \phi \alpha (a) \big) \cdot_E \big(  \alpha (b), \beta(v) + \phi \alpha (b) \big)\\
 =~& \big( \alpha(a) \cdot_A \alpha (b), ~ \underbrace{\alpha (a) \cdot_E \beta (v)}_{(A)} + \underbrace{\alpha (a) \cdot_E \phi \alpha (b)}_{(B)} + \underbrace{\beta (u) \cdot_E \alpha (b)}_{(C)} + \underbrace{\phi \alpha (a) \cdot_E \alpha (b) + \chi (\alpha (a), \alpha (b))}_{(B)} \\
 & \qquad \qquad \qquad + \underbrace{\beta(u) \cdot_B \beta (v)}_{(D)} + \underbrace{\beta(u) \cdot_B \phi \alpha (b)}_{(C)} + \underbrace{\phi \alpha (a) \cdot_B \beta(v)}_{(A)} + \underbrace{\phi \alpha(a) \cdot_B \phi \alpha (b)}_{(B)} \big)\\
 =~& \big( \alpha (a \cdot_A b),~ \underbrace{\beta (a \cdot_E v)}_{(A)} + \underbrace{\beta (\chi (a,b)) + \phi \alpha (a \cdot_A b)}_{(B)} + \underbrace{\beta (u \cdot_E b)}_{(C)} + \underbrace{\beta (u \cdot_B v)}_{(D)} \big) \\
 =~& \gamma \big(  a \cdot_A b, ~ a \cdot_E v + u \cdot_E b + \chi (a, b) + u \cdot_B v \big) \\
 =~& \gamma \big( (a,u) \cdot_E (b, v) \big) = \gamma (e \cdot_E f).
\end{align*}
which shows that $\gamma : E \rightarrow E$ is an algebra morphism. Moreover, 
\begin{align*}
(\gamma \circ U) (e) =~& (\gamma \circ U)(a,u) \\
=~& \gamma \big( R(a), S (u) + \Phi (a) \big) \\
=~& \big( \alpha R(a), \beta (S(u) + \Phi (a)) + \phi \alpha (R(a))  \big) \\
=~& \big( R \alpha (a), S \beta (u) + S \phi \alpha (a) + \Phi \alpha (a)  \big) \quad (\text{as } \Phi_{(\beta , \alpha)} - \Phi = S \phi - \phi R)\\
=~& U \big( \alpha (a), \beta (u) + \phi \alpha (a) \big)\\
=~& (U \circ \gamma)(a,u) = (U \circ \gamma)(e).
\end{align*}
As a summary, we get that $\gamma \in \mathrm{Aut}(E_U)$. Finally, from the definition of $\gamma$, it is easy to see that $\gamma|_B = \beta$ and $\overline{\gamma} = \pi \gamma s = \alpha$. Hence $\tau (\gamma) = (\gamma|_B, \overline{\gamma}) = (\beta, \alpha)$, which implies that the pair $(\beta, \alpha)$ is inducible.
 \end{proof}
 
\medskip
 
In the following, we will interpret the above theorem in terms of the Wells map in the context of Rota-Baxter algebras.
Let $0 \rightarrow B_S \xrightarrow{i} E_U \xrightarrow{\pi} A_R \rightarrow 0$ be a non-abelian extension of Rota-Baxter algebras. For any fixed section $s$, let $(\chi, \psi, \Phi)$ be the corresponding non-abelian $2$-cocycle. We define a map $\mathcal{W} : \mathrm{Aut}(B_S) \times \mathrm{Aut}(A_R) \rightarrow H^2_{nab}(A_R, B_S)$ by
\begin{align*}
\mathcal{W} ((\beta, \alpha)) := [(\chi_{(\beta, \alpha)}, \psi_{(\beta, \alpha)}, \Phi_{(\beta, \alpha)})- (\chi, \psi, \Phi) ].
\end{align*} 
Note that $\mathcal{W}$ is only a set map, not necessarily a group homomorphism. The map $\mathcal{W}$ is called the {\bf Wells map}. It is easy to see that the Wells map $\mathcal{W}$ doesn't depend on the choice of the section $s$.

Note that there is a natural action of the group $\mathrm{Aut}(B_S) \times \mathrm{Aut}(A_R)$ on the space $H^2_{nab}(A_R, B_S)$ given by
\begin{align*}
(\beta, \alpha) \cdot [(\chi, \psi, \Phi)] = [(\chi_{(\beta, \alpha)}, \psi_{(\beta, \alpha)}, \Phi_{(\beta, \alpha)})],
\end{align*}
for any $(\beta, \alpha) \in \mathrm{Aut}(B_S) \times \mathrm{Aut}(A_R)$ and $[(\chi, \psi, \Phi)] \in H^2_{nab}(A_R, B_S)$. With this action, the Wells map $\mathcal{W}$ is given by $\mathcal{W}((\beta, \alpha)) = (\beta, \alpha) \cdot [(\chi, \psi, \Phi)] - [(\chi, \psi, \Phi)]$, where $[(\chi, \psi, \Phi)] \in H^2_{nab} (A_R,B_S)$ is the element corresponding to the extension. This shows that the Wells map is simply a principal crossed homomorphism in the group cohomology of $\mathrm{Aut}(B_S) \times \mathrm{Aut}(A_R)$ with values in $H^2_{nab}(A_R, B_S)$.

\medskip

With the definition of the Wells map, Theorem \ref{thm-inducibility} can be rephrased as follows.

\begin{theorem}\label{thm-inducibility-2}
A pair $(\beta, \alpha) \in \mathrm{Aut}(B_S) \times \mathrm{Aut} (A_R)$ of Rota-Baxter automorphisms is inducible if and only if $\mathcal{W} ((\beta, \alpha)) = 0$. In other words, $\mathcal{W} ((\beta, \alpha))$ is an obstruction for the inducibility of the pair $(\beta, \alpha)$.
\end{theorem}


\medskip

\medskip

\noindent {\bf The fundamental exact sequence of Wells and some generalizations.} Given a non-abelian extension $0 \rightarrow B_S \xrightarrow{i} E_U \xrightarrow{\pi} A_R \rightarrow 0$ of Rota-Baxter algebras, here we show that the Wells map $\mathcal{W}$ fits into an exact sequence. This is analogue to the fundamental sequence of Wells in the context of Rota-Baxter algebras. We will further construct two generalizations and find their role in extending a Rota-Baxter automorphism $\beta \in \mathrm{Aut}(B_S)$ and lifting a Rota-Baxter automorphism $\alpha \in \mathrm{Aut}(A_R)$ to an automorphism in $\mathrm{Aut}(E_U)$.

\medskip

Let $\mathrm{Aut}_B^{B,A}(E_U) \subset \mathrm{Aut}_B (E_U)$ be the subgroup defined by 
\begin{align*}
\mathrm{Aut}_B^{B,A} (E_U) = \{ \gamma \in \mathrm{Aut}_B (E_U) ~|~ \tau (\gamma) = (\mathrm{id}_B, \mathrm{id}_A) \}.
\end{align*}

\begin{theorem}\label{thm-wells}(Fundamental sequence of Wells) There is an exact sequence
\begin{align*}
1 \rightarrow \mathrm{Aut}_B^{B,A} (E_U) \xrightarrow{\iota} \mathrm{Aut}_B (E_U) \xrightarrow{\tau} \mathrm{Aut}(B_S) \times \mathrm{Aut}(A_R) \xrightarrow{\mathcal{W}} H^2_{nab}(A_R,B_S).
\end{align*}
\end{theorem}

\begin{proof}
Note that the map $\iota : \mathrm{Aut}_B^{B,A} (E_U) \rightarrow \mathrm{Aut}_B (E_U)$ is the inclusion map. Hence the above sequence is exact at the first term.

Next, we show that the sequence is exact at the second term. Take an element $\gamma \in \text{ker} (\tau)$. Thus, we have
\begin{align*}
\tau (\gamma) = (\gamma|_B, \overline{\gamma}) = (\mathrm{id}_B, \mathrm{id}_A).
\end{align*}
This shows that $\gamma \in \mathrm{Aut}_B^{B,A} (E_U) \cong \mathrm{im} (\iota)$. Conversely, if $\gamma \in \mathrm{Aut}_B^{B,A}(E_U) \cong \mathrm{im}(\iota)$, then $\gamma \in \mathrm{ker}(\tau)$. Therefore, we have $\mathrm{ker}(\tau) = \mathrm{im}(\iota)$ which shows that the sequence is exact at the second term.

Next, we take an element $(\beta, \alpha) \in \mathrm{ker}(\mathcal{W})$. Therefore, it follows from Theorem \ref{thm-inducibility-2} that the pair $(\beta, \alpha)$ inducible. In other words, there exists an element $\gamma \in \mathrm{Aut}_B (E_U)$ such that $\tau (\gamma) = (\beta, \alpha),$ which shows that $(\beta, \alpha) \in \mathrm{im}(\tau)$. Conversely, if a pair $(\beta, \alpha) \in \mathrm{im} (\tau)$ then $(\beta, \alpha)$ is inducible. Hence by Theorem \ref{thm-inducibility-2}, we have $\mathcal{W}((\beta, \alpha)) =0$, which shows that $(\beta, \alpha) \in \text{ker}(\mathcal{W}).$ Therefore, we have $\text{ker}(\mathcal{W}) = \mathrm{im}(\tau)$. This proves that the sequence is exact at the third term. This completes the proof.
\end{proof}

\medskip



 We have seen earlier that the image of the Wells map can be seen as an obstruction for the inducibility of a pair of Rota-Baxter automorphisms. Besides the question about inducibility, there are some questions of interest. For example, one may ask the followings:
 
 \medskip

- When a Rota-Baxter automorphism $\alpha \in \mathrm{Aut}(A_R)$ can be lifted to an automorphism in $\mathrm{Aut}(E_U)$ fixing $B$ pointwise?

\medskip

- When a Rota-Baxter automorphism $\beta \in \mathrm{Aut}(B_S)$ can be extended to an automorphism in $\mathrm{Aut}(E_U)$ inducing the identity map on $A$?

\medskip

To answer these questions, we first introduce two subgroups $\mathrm{Aut}_B^B (E_U)$ and $\mathrm{Aut}_B^A (E_U)$ of the group $\mathrm{Aut}_B (E_U)$ as
\begin{align*}
\mathrm{Aut}_B^B (E_U) =~& \{ \gamma \in \mathrm{Aut}_B (E_U) ~|~ \gamma|_B = \mathrm{id}_B \},\\
\mathrm{Aut}_B^A (E_U )=~& \{ \gamma \in \mathrm{Aut}_B (E_U) ~|~ \overline{\gamma} = \mathrm{id}_A \}.
\end{align*}
Then there are obvious maps $\tau_1 : \mathrm{Aut}_B^B (E_U) \rightarrow \mathrm{Aut}(A_R)$ and $\tau_2 : \mathrm{Aut}_B^A (E_U) \rightarrow \mathrm{Aut}(B_S)$ given by
\begin{align*}
\tau_1 (\gamma) =~& \overline{\gamma}, \text{ for } \gamma \in \mathrm{Aut}_B^B (E_U), \\
\tau_2 (\gamma) =~& \gamma|_B, \text{ for } \gamma \in \mathrm{Aut}_B^A (E_U).
\end{align*}
We also define maps $\mathcal{W}_A : \mathrm{Aut}(A_R) \rightarrow H^2_{nab} (A_R, B_S)$ and $\mathcal{W}_B : \mathrm{Aut}(B_S) \rightarrow H^2_{nab} (A_R, B_S)$ by
\begin{align*}
\mathcal{W}_A (\alpha) =~& \mathcal{W} ((\mathrm{id}_B, \alpha)) =   [ (\chi_{(\mathrm{id}_B, \alpha)}, \psi_{(\mathrm{id}_B, \alpha)}, \Phi_{(\mathrm{id}_B, \alpha)}) - (\chi, \psi, \Phi)], \\
\mathcal{W}_B (\beta) =~& \mathcal{W} ((\beta, \mathrm{id}_A)) = [(\chi_{(\beta, \mathrm{id}_A)}, \psi_{(\beta, \mathrm{id}_A)}, \Phi_{(\beta, \mathrm{id}_A)}) - (\chi, \psi, \Phi)].
\end{align*}

In the following, we will construct two short exact sequences generalizing the fundamental sequence of Wells. The proofs of exactness for both the sequences are similar to the proof of Theorem \ref{thm-wells}. Hence we will not repeat it here.

\begin{theorem}\label{2thm-wells}
Let $0 \rightarrow B_S \xrightarrow{i} E_U \xrightarrow{\pi} A_R \rightarrow 0$ be a non-abelian extension of Rota-Baxter algebras. Then there are short exact sequences
\begin{align*}
1 \rightarrow \mathrm{Aut}_B^{B,A} (E_U) \xrightarrow{\iota} \mathrm{Aut}_B^B (E_U) \xrightarrow{\tau_1} \mathrm{Aut} (A_R) \xrightarrow{\mathcal{W}_A} H^2_{nab} (A_R, B_S), \\
1 \rightarrow \mathrm{Aut}_B^{B,A} (E_U) \xrightarrow{\iota} \mathrm{Aut}_B^A (E_U) \xrightarrow{\tau_2} \mathrm{Aut} (B_S) \xrightarrow{\mathcal{W}_B} H^2_{nab} (A_R, B_S).
\end{align*}
\end{theorem}

The following result is straightforward.

\begin{theorem}\label{2thm-wells-thm}
(i) A Rota-Baxter automorphism $\alpha \in \mathrm{Aut}(A_R)$ can be lifted to an automorphism in $\mathrm{Aut}(E_U)$ fixing $B$ pointwise if and only if $\mathcal{W}_A (\alpha) = 0.$

(ii) A Rota-Baxter automorphism $\beta \in \mathrm{Aut}(B_S)$ can be extended to an automorphism in $\mathrm{Aut}(E_U)$ inducing the identity on $A$ if and only if $\mathcal{W}_B (\beta) = 0.$
\end{theorem}

\section{Abelian extensions of Rota-Baxter algebras}\label{sec-5}
In this section, we show that the results of the previous sections fit with the abelian extensions of Rota-Baxter algebras. In particular, we classify the equivalence classes of abelian extensions in terms of the ordinary cohomology of Rota-Baxter algebras. We also discuss the inducibility problem and the fundamental sequence of Wells in the context of abelian extensions.

Let $A_R$ and $B_S$ be two Rota-Baxter algebras in which the associative product on $B$ is trivial. An {\bf abelian extension} of $A_R$ by $B_S$ is a short exact sequence 
\begin{align}
\xymatrix{
0 \ar[r] & B_S \ar[r]^i & E_U \ar[r]^\pi & A_R \ar[r] & 0
}
\end{align}
of Rota-Baxter algebras. The notion of equivalence between two abelian extensions of $A_R$ by $B_S$ can be defined as of Definition \ref{defin-nab}(ii).

For any section $s: A \rightarrow E$ of the map $\pi$, we define a map $\psi^s : (A \otimes B) \oplus (B \otimes A) \rightarrow B$ by
\begin{align*}
\psi^s (a,u) = s(a) \cdot_E u   ~~~ \text{ and } ~~~ \psi^s (u,a) = u \cdot_E s(a), \text{ for } a \in A, u \in B.
\end{align*}
Then it can be seen that $\psi^s$ gives rise to a Rota-Baxter bimodule structure on $B_S$ over the Rota-Baxter algebra $A_R$. Moreover, it doesn't depend on the choice of $s$ \cite{DasSKbimod}. This is called the induced Rota-Baxter bimodule.




Let $A_R$ be a Rota-Baxter algebra and $B_S$ be a Rota-Baxter bimodule over it. We consider $B_S$ as a Rota-Baxter algebra with the trivial associative product on $B$. We denote by $\mathrm{AbExt} (A_R, B_S)$ the set of all equivalence classes of abelian extensions of $A_R$ by $B_S$ so that the induced Rota-Baxter bimodule on $B_S$ coincides with the given one. In \cite{DasSKbimod} the authors parametrize $\mathrm{AbExt} (A_R, B_S)$ by the second (ordinary) cohomology of the Rota-Baxter algebra $A_R$ with coefficients in $B_S$. To recover this result, we first recall the (ordinary) cohomology of Rota-Baxter algebras.

Let $A_R$ be a Rota-Baxter algebra and $B_S$ be a Rota-Baxter bimodule over it. First consider the Hochschild cochain complex $\{ C^\bullet (A, B), \delta_\mathrm{Hoch} \}$ of the algebra $A$ with coefficients in the $A$-bimodule $B$, where $C^{n \geq 0} (A,B) = \mathrm{Hom}(A^{\otimes n}, B)$ and
\begin{align*}
(\delta_\mathrm{Hoch} f) (a_1, \ldots, a_{n+1} ) =~& a_1 \cdot f(a_2, \ldots, a_{n+1} ) + (-1)^{n+1} ~f(a_1, \ldots, a_n) \cdot a_{n+1} \\
&+ \sum_{i=1}^n (-1)^i f (a_1, \ldots, a_i \cdot_A a_{i+1}, \ldots, a_{n+1}),
\end{align*}
for $f \in C^n(A,B)$; $a_1, \ldots, a_{n+1} \in A.$ There is another cochain complex $\{ C^\bullet ({A}, {B}), \partial_\mathrm{Hoch} \}$ on the same cochain groups with the coboundary map
\begin{align*}
(\partial_\mathrm{Hoch} f) (a_1, \ldots, a_{n+1} ) =~& R(a_1) \cdot f(a_2, \ldots, a_{n+1} ) - S \big( a_1 \cdot f (a_2, \ldots, a_{n+1})  \big)\\
&+ (-1)^{n+1} ~ f(a_1, \ldots, a_n) \cdot R(a_{n+1}) - (-1)^{n+1} ~ S \big( f(a_1, \ldots, a_n) \cdot a_{n+1} \big) \\
&+ \sum_{i=1}^n (-1)^i f (a_1, \ldots, R(a_i) \cdot_A a_{i+1} + a_i \cdot_A R(a_{i+1}), \ldots, a_{n+1}),
\end{align*}
for $f \in C^n (A,B)$; $a_1, \ldots, a_{n+1} \in A$. As a byproduct of the above two cochain complexes, one may construct a new cochain complex $\{ C^\bullet_\mathrm{RBA} (A_R, B_S), \delta_\mathrm{RBA} \}$, where
\begin{align*}
C^0_\mathrm{RBA} (A_R, B_S) = B, ~~~~\quad C^1_\mathrm{RBA} (A_R, B_S) = C^1 (A,B),  ~~~~ \quad C^{n \geq 2}_\mathrm{RBA} (A_R, B_S) = C^n (A, B) \oplus C^{n-1} (A, B),
\end{align*}
and the differential $\delta_\mathrm{RBA} : C^\bullet_{\mathrm{RBA}} (A_R, B_S) \rightarrow C^{\bullet +1}_{\mathrm{RBA}} (A_R, B_S)$ given by
\begin{align*}
\delta_\mathrm{RBA} (u) =~& \delta_\mathrm{Hoch} (u),\\
\delta_\mathrm{RBA} ((f,g)) =~& \big(  \delta_\mathrm{Hoch} (f) ,~ \partial_\mathrm{Hoch} (g) + (-1)^n f \circ R^{\otimes n} - (-1)^n \sum_{i=1}^n S \circ f (R \otimes \cdots \otimes \underbrace{\mathrm{id}}_{i\mathrm{-th}} \otimes \cdots \otimes R) \big).
\end{align*}
The corresponding cohomology groups are called the (ordinary) cohomology of the Rota-Baxter algebra $A_R$ with coefficients in the Rota-Baxter bimodule $B_S$, and they are denoted by $H^\bullet_\mathrm{RBA} (A_R, B_S).$

\begin{theorem}\label{ab-ext-2ab}
Let $A_R$ be a Rota-Baxter algebra and $B_S$ be a Rota-Baxter bimodule over it. Then $$\mathrm{AbExt}(A_R,B_S) \cong H^2_{\mathrm{RBA}} (A_R, B_S).$$
\end{theorem}

\begin{proof}
Let $E_U$ be an abelian extension of $A_R$ by $B_S$. Since the associative product $\cdot_B$ is trivial, it follows from (\ref{nabb2}) that $\chi : A^{\otimes 2} \rightarrow B$ is a Hochschild $2$-cocycle on $A$ with values in the $A$-bimodule $B$, and the condition (\ref{nabb4}) implies that
\begin{align*}
\partial^1 (\Phi)(a, b) + \chi (R(a), R(b)) - S \big( \chi (R(a), b) + \chi (a, R(b))   \big) = 0, \text{ for all } a, b \in A.
\end{align*}
Thus, we have $\delta_\mathrm{RBA} (\chi, \Phi) (\chi , \Phi) = \big( \delta_\mathrm{Hoch}(\chi) , \partial^1 (\Phi) + \chi \circ (R \otimes R) - S \big( \chi \circ (R \otimes \mathrm{id} ) + \chi \circ (\mathrm{id} \otimes R) \big) \big) = 0$, which shows that $(\chi, \Phi) \in C^2_\mathrm{RBA} (A_R, B_S)$ is a $2$-cocycle. Moreover, if $E_U$ and $E'_{U'}$ are two equivalent abelian extensions, then one can check that the corresponding $2$-cocycles are cohomologous. Hence there is a well-defined map $\Lambda : \mathrm{AbExt}(A_R, B_S) \rightarrow H^2_\mathrm{RBA} (A_R, B_S)$.

The map $\Upsilon : H^2_\mathrm{RBA}(A_R, B_S) \rightarrow \text{AbExt} (A_R, B_S)$ on the other direction can be defined similar to Theorem \ref{non-ab-thm}. The maps $\Lambda$ and $\Upsilon$ are inverses to each other. Hence the proof.
\end{proof}


\medskip

Let $0 \rightarrow B_S \xrightarrow{i} E_U \xrightarrow{\pi} A_R \rightarrow 0$ be an abelian extension of Rota-Baxter algebras. Let $\psi$ denotes the given $A$-bimodule structure on $B$. Define
\begin{align*}
C_\psi = \big\{ (\beta, \alpha) \in \mathrm{Aut}(B_S) \times \mathrm{Aut} (A_R) ~|~ &\beta \psi (a,u) = \psi (\alpha(a), \beta (u)) \text{ and } \\
 &\beta \psi (u,a) = \psi ( \beta (u), \alpha(a)), \text{ for } a \in A, u \in B \big\}
\end{align*}
to be the set of all compatible pairs of Rota-Baxter automorphisms. Then $C_\psi$ is obviously a subgroup of $\mathrm{Aut}(B_S) \times \mathrm{Aut} (A_R)$. Given a section $s$ of the map $\pi$, let $(\chi, \Phi)$ be the $2$-cocycle corresponding to the abelian extension. For any pair $(\beta, \alpha) \in \mathrm{Aut}(B_S) \times \mathrm{Aut} (A_R)$, define a pair $(\chi_{(\beta, \alpha)}, \Phi_{(\beta, \alpha)})$ as 
\begin{align*}
\chi_{(\beta, \alpha)} (a, b) = \beta \circ \chi (\alpha^{-1}(a), \alpha^{-1}(b)) ~~~~ \text{ and } ~~~~ \Phi_{(\beta, \alpha)} (a) = \beta \circ \Phi (\alpha^{-1}(a)), \text{ for } a, b \in A.
\end{align*}
In general, $(\chi_{(\beta, \alpha)}, \Phi_{(\beta, \alpha)})$ is not a $2$-cocycle. However, if $(\beta, \alpha) \in C_\psi$, we can easily show that
\begin{align*}
\big( \delta_\mathrm{Hoch} (\chi_{(\beta, \alpha)})  \big) (a, b, c) = \beta \big( \underbrace{   \delta_\mathrm{Hoch} (\chi) }_{=0} (\alpha^{-1}(a), \alpha^{-1} (b), \alpha^{-1}(c) )  \big)= 0  ~~~ \text{ and }
\end{align*}
\begin{align*}
&\big(   \partial^1 (\Phi_{(\beta, \alpha)}) + \chi_{(\beta, \alpha)} \circ (R \otimes R) - S \big(  \chi_{(\beta, \alpha)} (R \otimes \mathrm{id}) + \chi_{(\beta, \alpha)} (\mathrm{id} \otimes R)  \big) \big) (a, b) \\
&= \beta \big( \underbrace{  \partial^1 (\Phi) + \chi \circ (R \otimes R) - S \big(  \chi (R \otimes \mathrm{id}) + \chi (\mathrm{id} \otimes R)  \big)}_{=0} \big) ( \alpha^{-1} (a), \alpha^{-1}(b) ) =0.
\end{align*}
Combining these, we get $\delta_{\mathrm{RBA}} (\chi_{(\beta, \alpha)} , \Phi_{(\beta, \alpha)}) = 0$ which shows that $(\chi_{(\beta, \alpha)} , \Phi_{(\beta, \alpha)})$ is a $2$-cocycle.



\begin{theorem}\label{ab-ext-ind-thm}
Let $0 \rightarrow B_S \xrightarrow{i} E_U \xrightarrow{\pi} A_R \rightarrow 0$ be an abelian extension of the Rota-Baxter algebra $A_R$ by the Rota-Baxter bimodule $B_S$. Then a pair $(\beta, \alpha) \in \mathrm{Aut}(B_S) \times \mathrm{Aut}(A_R)$ of Rota-Baxter automorphisms is inducible if and only if $(\beta, \alpha) \in C_\psi$ and the $2$-cocycles $(\chi_{(\beta, \alpha)}, \Phi_{(\beta, \alpha
)})$, $(\chi, \Phi)$ are cohomologous.
\end{theorem}

\begin{proof}
Note that the given $A$-bimodule structure $\psi$ is given by $\psi = \psi^s$. It follows from Theorem \ref{thm-inducibility} that $(\beta, \alpha)$ is inducible if and only if $\psi = \psi_{(\beta, \alpha)}$ and the $2$-cocycles $(\chi_{(\beta, \alpha)}, \Phi_{(\beta, \alpha
)})$, $(\chi, \Phi)$ are cohomologous. The result now follows as the condition $\psi = \psi_{(\beta, \alpha)}$ is equivalent to the fact that $(\beta, \alpha) \in C_\psi$. This completes the proof.
\end{proof}

The Wells map in the context of abelian extensions is given by $\mathcal{W} : C_\psi \rightarrow H^2_\mathrm{RBA} (A_R, B_S)$, where
\begin{align}\label{wells-map-ab}
\mathcal{W} ((\beta, \alpha)) = [(\chi_{(\beta, \alpha)}, \Phi_{(\beta, \alpha)}) - (\chi, \Phi)], \text{ for } (\beta, \alpha) \in C_\psi.
\end{align}
On the other hand, if $\gamma \in \mathrm{Aut}_B (E_U)$, then $\tau (\gamma) = (\gamma|_B, \overline{\gamma})$ is in $C_\psi$.

In \cite{barda} the authors considered the abelian extension of Lie algebras and interpret the automorphism group of the extension in terms of certain derivation space. Here we will generalize this result in the context of Rota-Baxter algebras.

\begin{proposition}
Let $0 \rightarrow B_S \xrightarrow{i} E_U \xrightarrow{\pi} A_R \rightarrow 0$ be an abelian extension of the Rota-Baxter algebra $A_R$ by a Rota-Baxter bimodule $B_S$. Then
\begin{align*}
\mathrm{Aut}_B^{B,A} (E_U) ~\cong ~ \mathrm{Der} (A_R, B_S) \text{ as groups.}
\end{align*}
\end{proposition}

\begin{proof}
We will only construct a bijective map between the above mentioned groups. Let $\gamma \in \mathrm{Aut}_B^{B,A} (E_U)$. Then for any $a \in A$, we have
\begin{align*}
\pi \big( (\gamma s - s)(a)  \big) = \pi \gamma s (a) - \pi s (a) = 0.
\end{align*}
This shows that $(\gamma s - s)(a) \in \mathrm{ker }(\pi) = \mathrm{im }(i) \cong B$. Therefore, one obtains a map $\overline{\gamma} : A \rightarrow B$ given by $\overline{\gamma} (a) = (\gamma s - s)(a)$, for $a \in A$. The map $\overline{\gamma}$ satisfies
\begin{align*}
\overline{\gamma}(a \cdot_A b) =~& \psi ( \overline{\gamma}(a), b) + \psi (a, \overline{\gamma}(b)), \text{ for } a, b \in A,\\
S \circ \overline{\gamma} =~& \overline{\gamma} \circ R.
\end{align*}
Thus, $\overline{\gamma} \in \mathrm{Der}(A_R, B_S)$. Hence there is a map $\Theta : \mathrm{Aut}_B^{B,A} (E_U) \rightarrow \mathrm{Der} (A_R, B_S)$, $\Theta (\gamma) = \overline{\gamma}$ which is also bijective. The map $\Theta$ is obviously a group homomorphism as
\begin{align*}
\Theta (\gamma \eta) = \gamma \eta s - s = \gamma (\eta s - s) + \gamma s - s = \gamma \overline{\eta} + \overline{\gamma} = \overline{\eta} + \overline{\gamma} ~(\text{as } \gamma|_B = \mathrm{id}_B).
\end{align*}
This completes the proof.
\end{proof}

\begin{theorem}\label{ab-ext-wells-thm}
Let $0 \rightarrow B_S \xrightarrow{i} E_U \xrightarrow{\pi} A_R \rightarrow 0$ be an abelian extension of the Rota-Baxter algebra $A_R$ by a Rota-Baxter bimodule $B_S$. Then there is a short exact sequence
\begin{align}\label{fun-ab}
1 \rightarrow \mathrm{Der} (A_R, B_S) \xrightarrow{\iota} \mathrm{Aut}_B(E_U) \xrightarrow{\tau} C_\psi \xrightarrow{\mathcal{W}} H^2_\mathrm{RBA} (A_R,B_S).
\end{align}
\end{theorem}

\medskip

An abelian extension $0 \rightarrow B_S \xrightarrow{i} E_U \xrightarrow{\pi} A_R \rightarrow 0$ of the Rota-Baxter algebra $A_R$ by a Rota-Baxter bimodule $B_S$ is said to be {\bf split} if there is a section $s$ (of the map $\pi$) such that $s: A_R \rightarrow E_U$ is a morphism of Rota-Baxter algebras. In this case, the Rota-Baxter algebra $E_U$ is isomorphic to the semidirect product $(A \oplus B)_{{R \oplus S}}$, where the associative product on $A \oplus B$ is given by
\begin{align*}
(a,u) \cdot_\ltimes (b, v) = \big( a \cdot_A b, \psi (a, v) + \psi (u, b) \big),
\end{align*}
for $(a,u), (b, v) \in A \oplus B$. Further, since $s : A_R \rightarrow E_U$ is a morphism of Rota-Baxter algebras, we have
\begin{align*}
\chi (a, b) = s(a) \cdot_E s(b) - s (a \cdot_A b) = 0 ~~~~ \text{ and } ~~~~ \Phi (a) = (Us - sR) (a) = 0.
\end{align*}
Therefore, we have $(\chi, \Phi) = 0$. As a consequence, it follows from (\ref{wells-map-ab}) that the Wells map vanishes identically. Thus, the exact sequence (\ref{fun-ab}) gives rise to the following sequence
\begin{align}\label{new-wells-ab}
0 \rightarrow \mathrm{Der}(A_R, B_S) \rightarrow \mathrm{Aut}_B (E_U) \xrightarrow{\tau} C_\psi \rightarrow 0.
\end{align}
It can be easily check that the above sequence is a split exact sequence. To see this, we define a map $t : C_\psi \rightarrow \mathrm{Aut}_B (E_U)$ by $t ((\beta, \alpha)) = t_{(\beta, \alpha)}$, where $t_{(\beta, \alpha)} : E_U \rightarrow E_U$ is given by $t_{(\beta, \alpha)} (a,u) = (\alpha (a) , \beta (u))$, for $(a,u) \in E \cong A \oplus B$. The map $t$ obviously satisfies $\tau t = \mathrm{id}_{C_\psi}$, which shows that $t$ is a section of the map $\tau$. The map $t$ is also a group homomorphism. Therefore, (\ref{new-wells-ab}) is a split exact sequence. As a summary, we obtain the following.

\begin{proposition}
Let $0 \rightarrow B_S \xrightarrow{i} E_U \xrightarrow{\pi} A_R \rightarrow 0$ be a split abelian extension of the Rota-Baxter algebra $A_R$ by a Rota-Baxter bimodule $B_S$. Then $\mathrm{Aut}_B (E_U) ~ \cong ~ C_\psi \ltimes \mathrm{Der}(A_R, B_S)$ as groups.
\end{proposition}

\section{Extensions of dendriform algebras and related problems}\label{sec-6}
In this section, we study non-abelian extensions of dendriform algebras in terms of non-abelian cohomology. Given a non-abelian extension, we also obtain the fundamental sequence of Wells regarding extensions and liftings of dendriform algebra automorphisms. We first start with the following definitions.

\begin{definition}
(i) Let $\mathcal{A}$ and $\mathcal{B}$ be two dendriform algebras. A {\bf non-abelian extension} of $\mathcal{A}$ by $\mathcal{B}$ is simply a short exact sequence of dendriform algebras
\begin{align}\label{dend-nab-seq}
\xymatrix{
0 \ar[r] & \mathcal{B} \ar[r]^i & \mathcal{E} \ar[r]^\pi & \mathcal{A} \ar[r] & 0.
}
\end{align}

(ii) Two non-abelian extensions $\mathcal{E}$ and $\mathcal{E}'$ are said to be {\bf equivalent} if there exists a morphism $\mathcal{E} \xrightarrow{\varphi} \mathcal{E}'$ of dendriform algebras which makes the following diagram commutative
\begin{align}\label{dend-nab-seq-eq}
\xymatrix{
0 \ar[r] & \mathcal{B} \ar@{=}[d] \ar[r]^i & \mathcal{E} \ar[d]^\varphi \ar[r]^\pi & \mathcal{A} \ar[r] \ar@{=}[d] & 0 \\
0 \ar[r] & \mathcal{B} \ar[r]_{i'} & \mathcal{E} \ar[r]_{\pi'} & \mathcal{A} \ar[r] & 0.
}
\end{align}
Let $\mathrm{Ext}(\mathcal{A}, \mathcal{B})$ denote the set of all equivalence classes of non-abelian extensions of $\mathcal{A}$ by $\mathcal{B}$.
\end{definition}

Next, we will discuss non-abelian cohomology of dendriform algebras. Let $C_n$ be the set of first $n$ natural numbers. We will often treat the elements of $C_n$ as some symbols (i.e. no operations on the elements of $C_n$ are allowed). Thus, for convenience, we write $C_n = \{ [1], [2], \ldots, [n] \}.$ There are some special maps $R_0^{1, \ldots, n, \ldots, 1} : C_{m+n-1} \rightarrow C_m$ and $R_i^{1, \ldots, n, \ldots, 1} : C_{m+n-1} \rightarrow {\bf k}[C_n]$ that are useful in the context of dendriform structures, given by\\

\begin{center}
\begin{tabular}{ |c|c|c|c| } 
\hline
$[r] \in C_{m+n-1}$ & $1 \leq r \leq i-1$ & $i \leq r \leq i+n-1$ & $i+n-1 \leq r \leq m+n-1$ \\ 
\hline
$R_0^{1, \ldots, n, \ldots, 1} [r]$ & $[r]$ & $[i]$ & $[i-n+1]$ \\ 
\hline
$R_i^{1, \ldots, n, \ldots, 1} [r]$ & $[1] + \cdots + [n]$ & $[r-i+1]$ & $[1] + \cdots + [n]$ \\ 
\hline
\end{tabular}
\end{center}

\medskip

\medskip

\noindent In the notations of the above maps $R_0^{1, \ldots, n, \ldots, 1}$ and $R_i^{1, \ldots, n, \ldots, 1}$, we have the superscripts as $m$ many numbers $\overbrace{1, \ldots, \underbrace{n}_{i\text{-th}}, \ldots, 1}^{m ~\mathrm{ many}}$ with all $1$ except in the $i$-th place where we have $n$. Let $(\mathcal{A}, \prec_\mathcal{A}, \succ_\mathcal{A})$ be a dendriform algebra. Then there is a map $\pi_\mathcal{A} : {\bf k}[C_2] \otimes \mathcal{A}^{\otimes 2} \rightarrow \mathcal{A}$ given by 
\begin{align*}
\pi_\mathcal{A} ([1]; a, b) = a \prec_\mathcal{A} b ~~~ \text{ and } ~~~ \pi_\mathcal{A} ([2]; a, b) = a \succ_\mathcal{A} b, ~\text{ for } a, b \in \mathcal{A}.
\end{align*}

\medskip

\begin{definition}\label{defin-dend-non2}
(i) Let $\mathcal{A}$ and $\mathcal{B}$ be two dendriform algebras. A {\bf non-abelian $2$-cocycle} on $\mathcal{A}$ with values in $\mathcal{B}$ is given by a pair $(\chi, \psi)$ in which
\begin{align*}
\chi : {\bf k}[C_2] \otimes \mathcal{A}^{\otimes 2} \rightarrow  \mathcal{B} ~~~ \text{ and } ~~~ \psi : {\bf k}[C_2] \otimes \big( (\mathcal{A} \otimes \mathcal{B}) \oplus (\mathcal{B} \otimes \mathcal{A}) \big) \rightarrow \mathcal{B}
\end{align*}
are linear maps satisfying
\begin{align}\label{nabb1-dend}
\begin{cases}
&\psi \big(   R_0^{1,2} [r]; a ,~\psi (R_2^{1,2} [r]; b, u) \big) = \psi \big( R_0^{2,1} [r];  \pi_\mathcal{A} (R_1^{2,1} [r]; a,b), u  \big) +\pi_\mathcal{B} \big(  R_0^{2,1}[r]; \chi (R_1^{2,1}[r]; a, b), u \big),\\
&\psi \big(   R_0^{1,2} [r]; a ,~\psi (R_2^{1,2} [r]; u, b) \big) = \psi \big( R_0^{2,1} [r];  \psi (R_1^{2,1} [r]; a,u), b  \big), \\
&\psi \big( R_0^{2,1} [r];  \psi (R_1^{2,1} [r]; u,a), b  \big) = \psi \big(   R_0^{1,2} [r]; u ,~\pi_\mathcal{A} (R_2^{1,2} [r]; a, b) \big) + \pi_\mathcal{B} \big(  R_0^{1,2} [r]; u ,~\chi (R_2^{1,2} [r]; a, b)  \big),
\end{cases}
\end{align}
\begin{align}\label{nabb-2-dend}
\psi \big( R_0^{1,2} [r]; a, \chi (R_2^{1,2} [r]; b, c)   \big)   - \chi \big(  R_0^{2,1} [r];& \pi_\mathcal{A} ( R_1^{2,1}[r]; a, b), c  \big) + \chi \big( R_0^{1,2} [r]; a, \pi_\mathcal{A} (R_2^{1,2} [r]; b, c)   \big)   \\
&-  \psi \big(  R_0^{2,1} [r]; \chi ( R_1^{2,1}[r]; a, b), c  \big) = 0, \nonumber
\end{align}
for $a, b, c \in \mathcal{A}$, $u \in \mathcal{B}$ and $[r] \in C_3.$

(ii) Two non-abelian $2$-cocycles $(\chi, \psi)$ and $(\chi', \psi')$ are said to be {\bf equivalent} if there exists a linear map $\phi : \mathcal{A} \rightarrow \mathcal{B}$ that satisfies
\begin{align}
\begin{cases}
\psi ([r]; a, u) - \psi' ([r]; a, u) = \pi_\mathcal{B} \big( [r] ; \phi (a), u  \big), \\
\psi ([r]; u, a) - \psi' ([r]; u, a) = \pi_\mathcal{B} \big( [r]; u, \phi (a) \big),
\end{cases}
\end{align}
\begin{align}
\chi ([r]; a, b) - \chi' ([r]; a, b) = \psi' \big([r]; a , \phi (b) \big) + \psi' \big( [r]; \phi (a), b  \big) - \phi \big(   \pi_\mathcal{A} ([r]; a, b) \big) + \pi_\mathcal{B} \big( [r]; \phi(a), \phi(b)  \big),
\end{align}
for $a, b \in \mathcal{A}$, $u \in \mathcal{B}$ and $[r] \in C_2$. In this case, we write $(\chi, \psi) \overset{\phi}{\sim} (\chi', \psi')$.
\end{definition}

We denote the set of all equivalence classes of non-abelian $2$-cocycles by $H^2_{nab} (\mathcal{A}, \mathcal{B})$. This is called the {\bf non-abelian cohomology} of the dendriform algebra $\mathcal{A}$ with values in the dendriform algebra $\mathcal{B}$. When $\mathcal{B}$ is equipped with the trivial dendriform structure (i.e. when $\pi_\mathcal{B} = 0$), then a non-abelian $2$-cocycle $(\chi, \psi)$ is equivalent to having a representation $\psi$ of the dendriform algebra $\mathcal{A}$ on the module $\mathcal{B}$, and an ordinary $2$-cocycle $\chi$ on the dendriform algebra $\mathcal{A}$ with coefficients in the representation $\mathcal{B}$. Further, two non-abelian $2$-cocycles $(\chi, \psi)$ and $(\chi', \psi')$ are equivalent if and only if $\psi = \psi'$ and the ordinary $2$-cocycles $\chi, \chi'$ are cohomologous in the ordinary sense. Thus, in this case, the non-abelian cohomology $H^2_{nab}(\mathcal{A}, \mathcal{B})$ coincides with the ordinary dendriform cohomology.

\begin{theorem}\label{non-ab-dend-thm}
Let $\mathcal{A}$ and $\mathcal{B}$ be two dendriform algebras. Then there is a one-to-one correspondence between the space $\mathrm{Ext}(\mathcal{A}, \mathcal{B})$ and the group $H^2_{nab}(\mathcal{A}, \mathcal{B}).$
\end{theorem}

\begin{proof}
Let $\mathcal{E}$ be a non-abelian extension of $\mathcal{A}$ by $\mathcal{B}$ (Definition \ref{dend-nab-seq}(i)). For any section $s: \mathcal{A} \rightarrow \mathcal{E}$ of the map $\pi$, we define maps 
\begin{align*}
\chi^s : {\bf k}[C_2] \otimes \mathcal{A}^{\otimes 2} \rightarrow  \mathcal{B} ~~~ \text{ and } ~~~ \psi^s : {\bf k}[C_2] \otimes \big( (\mathcal{A} \otimes \mathcal{B}) \oplus (\mathcal{B} \otimes \mathcal{A}) \big) \rightarrow \mathcal{B} ~~\text{ by }
\end{align*}
\begin{align*}
\chi^s ([r]; a, b) =~& \pi_\mathcal{E} ([r]; s(a) , s(b)) ~-~ s \big( \pi_\mathcal{A}([r];a, b) \big),\\
\psi^s ([r] ; a, u) =~& \pi_\mathcal{E} ([r]; s(a) , u), \\
\psi^s ([r]; u, a) =~& \pi_\mathcal{E} ([r]; u, s(a)),
\end{align*}
for $a, b \in \mathcal{A}$, $u \in \mathcal{B}$ and $[r] \in C_2$. Then it is not hard to verify that the maps $\chi^s, \psi^s$ satisfy the identities (\ref{nabb1-dend}) and (\ref{nabb-2-dend}). This shows that the pair $(\chi^s, \psi^s)$ is a nonabelian $2$-cocycle. Let $t : \mathcal{A} \rightarrow \mathcal{E}$ be any other section of the map $\pi$ and $(\chi^t, \psi^t)$ be the corresponding non-abelian $2$-cocycle. Define a map $\phi : \mathcal{A} \rightarrow \mathcal{B}$ by $\phi := s - t$. Then we have
\begin{align*}
\psi^s ([r]; a, u) - \psi^t ([r]; a, u) = \pi_\mathcal{E} ([r]; s(a), u) - \pi_\mathcal{E} ([r]; t(a) , u) = \pi_\mathcal{B} ([r]; \phi (a), u), \\
\psi^s ([r];  u,a) - \psi^t ([r]; u,a) = \pi_\mathcal{E} ([r]; u, s(a) ) - \pi_\mathcal{E} ([r];  u, t(a)) = \pi_\mathcal{B} ([r];  u, \phi (a)),
\end{align*}
\begin{align*}
&\chi^s ([r]; a, b) - \chi^t ([r]; a, b) \\
&= \pi_\mathcal{E} \big( [r]; s(a), s(b) \big) - s \big(  \pi_\mathcal{A} ([r]; a,b) \big) - \pi_\mathcal{E} \big( [r]; t(a), t(b) \big) + t \big(  \pi_\mathcal{A} ([r]; a,b) \big)\\
&= \pi_\mathcal{E} \big( [r]; (\phi + t) (a), (\phi + t)(b)  \big) - \phi \big(  \pi_\mathcal{A} ([r]; a,b) \big) - \pi_\mathcal{E} \big( [r]; t(a), t(b) \big) \\
&= \psi^t ([r]; a, \phi (b) ) + \psi^t ([r]; \phi(a), b)  - \phi (\pi_\mathcal{A} ([r];a, b)) + \pi_\mathcal{E} ([r]; \phi (a), \phi (b)),
\end{align*}
for $a, b \in \mathcal{A}$ and $u \in \mathcal{B}$. Thus $(\chi^s, \psi^s) \overset{\phi}{\sim} (\chi^t, \psi^t)$. Hence $(\chi^s, \psi^s)$ and $(\chi^t, \psi^t)$ induces the same element in $H^2_{nab}(\mathcal{A}, \mathcal{B})$.

If $\mathcal{E}, \mathcal{E}'$ are two equivalent non-abelian extensions and $s: \mathcal{A} \rightarrow \mathcal{E}$ is a section of the first extension, then $s' := \phi  s$ is a section of the second extension. If $(\chi^{s'}, \psi^{s'})$ is the non-abelian $2$-cocycle corresponding to the second extension and its section $s'$, then by a simple calculation one can show that
$\chi^{s'} = \chi^s$ and
$\psi^{s'} = \psi^s.$
Hence $(\chi^{s'}, \psi^{s'}) = (\chi^s, \psi^s)$. As a conclusion, we obtain a map $\Lambda : \mathrm{Ext}(\mathcal{A}, \mathcal{B}) \rightarrow H^2_{nab}(\mathcal{A}, \mathcal{B}).$ 

\medskip

On the other hand, if $(\chi, \psi)$ is a non-abelian $2$-cocycle as of Definition \ref{defin-dend-non2}, we define a dendriform algebra structure on the direct sum $\mathcal{A} \oplus \mathcal{B}$ given by
\begin{align*}
(a,u) \prec_{\chi, \psi} (b, v) = \big( a \prec_\mathcal{A} b,~ \psi ([1]; a, v) + \psi ([1] ; u, b) + \chi ([1]; a, b) + u \prec_\mathcal{B} v   \big),\\
(a,u) \succ_{\chi, \psi} (b, v) = \big(  a \succ_\mathcal{A} b,~ \psi ([2]; a, v) + \psi ([2] ; u, b) + \chi ([2]; a, b) +  u \succ_\mathcal{B} v   \big).
\end{align*}
(Note that the above two products form a dendriform algebra structure as $\chi, \psi$ satisfy the identities (\ref{nabb1-dend}) and (\ref{nabb-2-dend})). We denote this dendriform algebra simply by $\mathcal{A}\oplus_{\chi, \psi} \mathcal{B}$. Moreover, it is a non-abelian extension of $\mathcal{A}$ by $\mathcal{B}$ with the obvious structure maps. Further, let $(\chi, \psi)$ and $(\chi', \psi')$ be two equivalent non-abelian $2$-cocycles (see Definition \ref{defin-dend-non2}(ii)). We define a map $\varphi : \mathcal{A} \oplus \mathcal{B} \rightarrow \mathcal{A} \oplus \mathcal{B}$ by
\begin{align*}
\varphi (a, u) = (a, u + \phi (a)), \text{ for } (a, u) \in \mathcal{A} \oplus \mathcal{B}.
\end{align*}
Then it is straightforward to verify that
\begin{align*}
\varphi ((a, u)  \prec_{\chi, \psi} (b, v) ) = \varphi (a,u) \prec_{\chi', \psi'} \varphi (b,v) ~~ \text{ and } ~~ \varphi ((a, u)  \succ_{\chi, \psi} (b, v) ) = \varphi (a,u) \succ_{\chi', \psi'} \varphi (b,v). 
\end{align*}
This show that $\varphi : \mathcal{A} \oplus_{\chi, \psi} \mathcal{B} \rightarrow \mathcal{A} \oplus_{\chi', \psi'} \mathcal{B}$ is a morphism of dendriform algebras, which induces an equivalence between non-abelian extensions. As a result, we obtain a map $\Upsilon : H^2_{nab}(\mathcal{A}, \mathcal{B}) \rightarrow \mathrm{Ext}(\mathcal{A}, \mathcal{B})$. The result now follows as the map $\Lambda$ and $\Upsilon$ are inverses to each other.
\end{proof}

\medskip

Let $0 \rightarrow \mathcal{B} \xrightarrow{i} \mathcal{E} \xrightarrow{\pi} \mathcal{A} \rightarrow 0$ be a non-abelian extension of dendriform algebras. For any section $s : \mathcal{A} \rightarrow \mathcal{E}$ of the map $\pi$, let $(\chi, \psi)$ be the non-abelian $2$-cocycle on $\mathcal{A}$ with values in $\mathcal{B}$ corresponding to the above extension. Let $(\beta, \alpha) \in \mathrm{Aut} (\mathcal{B}) \times \mathrm{Aut} (\mathcal{A})$ be any pair of dendriform algebra automorphisms. Consider the non-abelian $2$-cocycle $(\chi_{(\beta, \alpha)}, \psi_{(\beta, \alpha)})$, where
\begin{align*}
\chi_{(\beta, \alpha)} ([r]; a, b) = \beta \circ \chi ([r]; \alpha^{-1}(a), \alpha^{-1}(b))  ~~ \text{ and } ~~ \begin{cases} \psi_{(\beta, \alpha)} ([r]; a, u) = \beta \circ \psi ([r]; \alpha^{-1} (a), \beta^{-1}(u)),\\
\psi_{(\beta, \alpha)} ([r]; u, a) = \beta \circ \psi ([r];  \beta^{-1}(u), \alpha^{-1} (a)),
\end{cases}
\end{align*}
for $a, b \in \mathcal{A}$, $u \in \mathcal{B}$ and $[r] \in C_2$. Then the inducibility theorem in the context of dendriform algebras can be given as follows. The proof of the result is similar to the proof of Theorem \ref{thm-inducibility}, hence we do not repeat here.

\begin{theorem}
Let $0 \rightarrow \mathcal{B} \xrightarrow{i} \mathcal{E} \xrightarrow{\pi} \mathcal{A} \rightarrow 0$ be a non-abelian extension of dendriform algebras. A pair $(\beta, \alpha) \in \mathrm{Aut} (\mathcal{B}) \times \mathrm{Aut} (\mathcal{A})$ of dendriform algebra automorphisms is inducible if and only if the non-abelian $2$-cocycles $(\chi, \psi)$ and $(\chi_{(\beta, \alpha)}, \psi_{(\beta, \alpha)})$ are equivalent.
\end{theorem}

With all the above notations, the Wells map $\mathcal{W} : \mathrm{Aut}(\mathcal{B}) \times \mathrm{Aut}(\mathcal{A}) \rightarrow H^2_{nab} (\mathcal{A}, \mathcal{B})$ is defined by
\begin{align*}
\mathcal{W}((\beta, \alpha)) = [ \chi_{(\beta, \alpha)}, \psi_{(\beta, \alpha)}) - (\chi, \psi) ],
\end{align*}
where $(\chi, \psi)$ is the non-abelian $2$-cocycle corresponding to the non-abelian extension of dendriform algebras. A pair $(\beta, \alpha) \in \mathrm{Aut} (\mathcal{B}) \times \mathrm{Aut} (\mathcal{A})$ of dendriform algebra automorphisms is inducible if and only if $\mathcal{W} ((\beta, \alpha)) = 0$. Further, the fundamental exact sequence of Wells is given by
\begin{align*}
0 \rightarrow \mathrm{Aut}_\mathcal{B}^{\mathcal{B}, \mathcal{A}} (\mathcal{E}) \xrightarrow{\iota} \mathrm{Aut}_\mathcal{B} (\mathcal{E}) \xrightarrow{\tau} \mathrm{Aut}(\mathcal{B}) \times  \mathrm{Aut}(\mathcal{A}) \xrightarrow{\mathcal{W}} H^2_{nab} (\mathcal{A}, \mathcal{B}), 
\end{align*}
where all the notations have meanings similar to the context of Rota-Baxter algebras.




\end{document}